\documentclass[a4paper,11pt,reqno]{amsart}
\usepackage[T1]{fontenc}
\usepackage[utf8]{inputenc}
\usepackage[margin=1.1in]{geometry}
\usepackage{lmodern}
\usepackage[english]{babel}
\usepackage{graphicx}
\usepackage{amsmath}
\usepackage{amsfonts}
\usepackage{amssymb} 
\usepackage{amsthm}
\usepackage{bbm}
\usepackage{bm} 
\usepackage{lipsum}
\usepackage{xcolor}
\usepackage{url}
\usepackage{hyperref}
\usepackage{stmaryrd}

\allowdisplaybreaks

\theoremstyle{plain}
\newtheorem{thm}{Theorem}
\newtheorem{lem}[thm]{Lemma}
\newtheorem{prop}[thm]{Proposition}

\newtheorem{defn}[thm]{Definition}
\newtheorem{rem}[thm]{Remark}
\newtheorem{ass}{Assumption}

\theoremstyle{definition}
\newtheorem{exmp}[thm]{Example}

\newtheorem{innercustomgeneric}{\customgenericname}
\providecommand{\customgenericname}{}
\newcommand{\newcustomtheorem}[2]{%
  \newenvironment{#1}[1]
  {%
   \renewcommand\customgenericname{#2}%
   \renewcommand\theinnercustomgeneric{##1}%
   \innercustomgeneric
  }
  {\endinnercustomgeneric}
}

\newcustomtheorem{customthm}{Theorem}

\def \be {\begin{equation}}
\def \ee {\end{equation}}

\def \E {\mathbb{E}}
\def \P {\mathbb{P}}
\def \R {\mathbb{R}}
\def \ES {\llbracket d \rrbracket}
\def \NN {\llbracket N \rrbracket}
\def \ESd {\llbracket d-1 \rrbracket}
\def \Law {\mathrm{Law}}
\def \N {\mathcal{N}}

\def \Int {\mathrm{Int}(S_d)}
\def \Intx {\mathrm{Int}(\hat{S}_d)}

\renewcommand{\phi}{\varphi}
\renewcommand{\epsilon}{\varepsilon}
\renewcommand{\tilde}{\widetilde}
\renewcommand{\hat}{\widehat}
\renewcommand{\bar}{\overline}

\begin{document}

\title[Finite state $N$-agent and mean field control problems]{Finite state $N$-agent and mean field control problems}

\author{Alekos Cecchin}

\address[A. Cecchin]
{\newline \indent Centre de Math\'ematiques Appliqu\'ees, \'Ecole Polytechnique
\newline
\indent Route de Saclay, 
91124 Palaiseau Cedex, France
\newline }
\email{alekos.cecchin@polytechnique.edu}
\urladdr{http://sites.google.com/view/alekoscecchin}
\thanks{
This research benefited from the support of the project ANR-16-CE40-0015-01 on ``Mean Field Games'', 
LABEX Louis Bachelier Finance and Sustainable Growth - project 
ANR-11-LABX-0019, under the Investments for the Future program (in accordance
with Article 8 of the Assignment Agreements for Communication 
Assistance), ECOREES ANR Project, FDD Chair and Joint Research Initiative
FiME in partnership with Europlace Institute of Finance.
\newline
 The author is grateful to  Fran\c cois Delarue and Charles Bertucci for helpful comments and suggestions, as well as to  two anonymous referees for their remarks which helped to improve the paper. 
}

\date{\today}

\begin{abstract}
We examine mean field control problems  on a finite state space, in continuous time and over a finite time horizon. We characterize the value function of the mean field control problem as the unique viscosity solution of a Hamilton-Jacobi-Bellman equation in the simplex. In absence of any convexity assumption, we exploit this characterization to prove convergence, as $N$ grows, of the value functions of the centralized $N$-agent optimal control problem to the limit mean field control problem  value function, with a convergence rate of order $\frac{1}{\sqrt{N}}$. Then, assuming convexity, we show that the limit value function is smooth and establish propagation of chaos, i.e.  convergence of the $N$-agent optimal trajectories to the unique limiting optimal trajectory, with an explicit rate.
\end{abstract}

\maketitle
 
\tableofcontents

\section{Introduction}

Mean field control problems (MFCP), also called control of McKean-Vlasov equations, can be interpreted as limit of cooperative $N$-agent games, as the number of players tends to infinity. 
Such agents have a common cost to minimize and the minimizers are also called \emph{Pareto} equilibria; alternatively, we could think of a social planner that minimizes an average cost.  

We investigate $N$-agent optimization and mean field control problems in continuous time over a finite time horizon, with dynamics belonging to a finite state space $\{1,\dots,d\}$. More precisely, the $N$ agents $\bm{X}=(X^1,\dots,X^N)$ follow the dynamics 
\[
\P(X^k_{t+h}=j |\bm{X}_t=\bm{x} )= Q_{x_k,j}(t,\beta_k(t,\bm{x}), \mu^N_{\bm{x}}) h +o(h),
\]
where $\mu^N$ is the empirical measure and the controls (here in feedback form) $\bm{\beta}=(\beta_1,\dots,\beta_N)$ are chosen in order to minimize the common cost
\[
J^N(\bm{\beta}) := \frac1N \sum_{k=1}^N \E\left[ \int_0^T   f(t,X^k_t, \beta_k (t, \bm{X}_t),\mu^N_t)
 dt + g(X^k_T,\mu^N_T)  \right].
\]
Assuming that controls depend only on the empirical measure, the usual propagation of chaos arguments suggest that the limit of this $N$-agent optimization, as $N$ grows, consists in a single player which evolves according to 
\[
\P(X_{t+h}=j | X_t=i)= Q_{i,j}(t, \alpha^i(t), \Law(X_t)) h +o(h)
\]
and aims at minimizing
\[
J(\alpha) :=  \E\left[ \int_0^T   f(t, X_t, \alpha (t, X_t),  \Law(X_t) )  dt + g(X_T,\Law(X_T))  \right].
\]
Our goal is to study in detail the $N$-agent optimization and the mean field control problem and thus prove  convergence of the former to the latter, 
the main result being to provide an explicit convergence rate.

Such problems are studied so far mainly for  continuous state space and diffusion-based dynamics. In such situation, the limiting MFCP can be analyzed in two different ways, by considering either open-loop or (Markovian) feedback controls. In the first case, a version of the Pontryagin principle is derived in \cite{CarDelMKV}, which leads to study a forward-backward system of (It\^o-type) SDEs of McKean-Vlasov type.  While the case of feedback controls is analyzed in \cite{PhamWei, PhamWei2}, where the MFCP is reformulated in a deterministic way as the optimal control of the Fokker-Planck equation, which permits to derive a dynamic programming principle  and then a Hamilton-Jacobi-Bellman (HJB) equation for the value function, written in the Wasserstain space of probability measures. See also \cite{LaurierePironneau, Bensoussan} for previous ideas in this direction and \cite{BCP, DPT1} for more general versions of the dynamic programming principle. 
We refer to \cite[Chapter 6]{CarDel1} for a comparison of the two approaches.
Mean field control problems arise also in the study of potential mean field games, see e.g. \cite{CGPT,BrianiCarda}, where it is shown that the mean field game system represents the necessary conditions for optimality of a suitable mean filed control problem.

The question of convergence of the $N$-agent optimization to the MFCP, still in the diffusion setting, has been analyzed mainly in two ways. The first consists in showing that the set of (relaxed open-loop) optimizers of the $N$-agent optimization is precompact and then that the limit points are supported on optimizers of the MFCP. We remark that the optimizer is non-unique in general, but is unique under additional convexity assumptions; we return to this point below.  This strategy is employed first in \cite{Lacker}, then in \cite{FLOS} for deterministic dynamics and, more recently, in \cite{DPT} for more general dynamics with a common noise and in \cite{djete} for problems with interaction also through the law of the control.
Clearly, a convergence rate for the convergence of the value functions can not be proved using compactness arguments. 
The other way is to prove convergence via the system of FBSDEs, in case the limit solution is unique. In \cite{CarDelMKV}, the value functions are shown to converge, with a suitable convergence rate, assuming that the cost $f$ is convex in $(a,x,m)$ and $g$ is convex in $(x,m)$; see also \cite[\S 6.1.3]{CarDel2}.
  Moreover, a propagation of chaos property is also proved, i.e. the prelimit (unique) optimal trajectories are shown to converge to the (unique) limit optimal trajectory, with a convergence rate (actually, such result is not stated in this way, but can be immediately derived from the proof of Theorem 6.1 therein). More recently, this method has been applied to problems with interaction through the law of the control in \cite{LauriereTangpi}.
In both cases, as a consequence of the convergence of the value functions,  an optimal control for the MFCP is shown to be $\epsilon_N$-optimal for the $N$-agent optimization, with $\lim_N \epsilon_N =0$, with the same limitations just explained (no convergence rate in the first case and convexity required in the second).

Here, in the finite state setting, we first analyze the MFCP with feedback controls. We rewrite it as a deterministic control problem for the Fokker-Planck equation and show that its value function $V$ is the unique viscosity solution of the corresponding HJB equation, stated in the $d$-dimensional simplex. Then we examine the convergence problem: we stress that convergence can be understood both in terms of value functions and of optimal trajectories. 
Our main result is to  show that the value function $V^N$ of the $N$-agent optimization converges to $V$, with a convergence rate of order $1/\sqrt{N}$, in absence of any convexity assumption. 
As explained above, a result of this type is not available for diffusion-based models.
As a consequence of this convergence, we also prove that any optimal control for the MFCP is $\frac{C}{\sqrt{N}}$-optimal for the $N$-agent optimization.
A similar result is proved also in \cite{kol} with different methods and without interpreting the prelimit model as a $N$-agent optimization; see Remark \ref{remk} for the details. In the discrete time setting, the MFCP and the related convergence of the $N$ agent are investigated in full generality in \cite{motte}.

Our main novelty consists in the method of the proof of the main result, which is different  with respect to the two explained above and we believe can be of interest.
This is based on the viscosity solution characterization of $V$: indeed, we will see that the ODE satisfied by $V^N$ can be seen as a finite difference scheme for the HJB equation satisfied by $V$.
Notably, $V$ is not differentiable as we do not assume any convexity of the costs, neither in $a$ nor in $m$.  
Finite difference schemes for viscosity solutions have been investigated by many authors in the last decades. Being impossible to give a complete bibliography on the subject, we would like to mention two papers which inspired our proof of convergence and that established, in particular, a rate of convergence. 
 The first is \cite{ICD}, which studied a semidiscrete approximation scheme for the HJB equation of infinite horizon control problem with discount; see also the book
\cite[Sec. VI.1]{BCD}. While the second is \cite{Sou}, which analyzed a finite difference scheme for a general time-dependent Hamilton-Jacobi equation; see also \cite{CL}.

We also study the propagation of chaos property for the optimal trajectories, in case the limit is unique. If the value function $V$ is sufficiently smooth, i.e. in $\mathcal{C}^{1,1}$, then the MFCP is uniquely solvable and we  prove that the prelimit (unique) optimal empirical measures converge to the limit deterministic optimal flow, with a suitable convergence rate. We give also sufficient conditions for which $V\in \mathcal{C}^{1,1}$: these are the standard convexity assumptions. We remark that also these smoothness and propagation of chaos results seem to be new in the study of MFCP. 
Moreover, it is worth saying that we do not treat here neither problems with a common noise nor with interaction through the law of the control; these are left to future work. 

Finally, we mention that convergence results have been obtained also for the opposite regime of mean field games. In that case, players are non-cooperative in the prelimit $N$-player game and the notion of optimality is that of Nash equilibrium, which highly depends on the set of admissible startegies that is considered. This makes the convergence analysis more difficult, expecially in case limiting mean field game solutions are non-unique; some references are \cite{CDLL, DelFT, Fischer, Lacker1, Lacker2} for diffusion-based models and 
\cite{BCCD-N, BC, CDFP, CP, gomes} for finite state space.

The rest of the paper is organized as follows. In Section \ref{sec2}, we collect our main results: after introducing the notations and assumptions that will be in force, we define properly the $N$-agent optimization and show the equivalence with the mean field formulation, then we present the MFCP with its well-posedness result (Theorem \ref{thm:mainmfc}) and thus we state the convergence theorems. In Section \ref{sec3}, we examine the MFCP and prove first well-posedness of viscosity solutions and then, under additional convexity assunptions, well-posedness of classical solutions. 
Notably, we establish a comparison principle (Theorem \ref{thmcomparison}) for viscosity solutions on the interior of the simplex --without boundary conditions, by exploiting the invariance of the domain--, which is a new result in the theory. 
Finally, Section \ref{sec4} contains the proofs of the convergence results (Theorems \ref{thmconvV} and \ref{thmchaos}): first the convergence of the value functions via viscosity solutions and then, assuming $V\in \mathcal{C}^{1,1}$, the propagation of chaos property.

\section{Main Results}
\label{sec2}

\subsection{Notation}

We denote $\ES = \{1,\dots,d\}$ and let 
\[
S_d := \bigg\{ (m_1, \dots, m_d) \in \R^d : \quad m_i  \geq 0  \quad \forall i\in\ES, \quad 
\sum_{i=1}^d m_i =1 \bigg\} 
\] 
be the $(d-1)$-dimensional simplex, endowed with the euclidean norm $|\cdot |$ in $\R^d$.   
  We denote by $\langle \cdot, \cdot \rangle$ the scalar product in $\R^d$ and, for a matrix $(Q_{i,j})_{i,j\in\ES}$, we let $Q_{i,\bullet}$ be the row $i$. We denote the elements of the simplex by $m$, while $\mu$  denotes processes with values in the simplex.
 
 The simplex can be viewed as a subset of $\R^{d-1}$, by expressing the last coordinate as $m_d=1-\sum_{j=1}^{d-1} m_j$: denote then 
 \[
 \hat{S}_d := \bigg\{ (x_1, \dots, x_{d-1}) \in \R^{d-1} : \quad x_j  \geq 0  \quad \forall j=1,\dots, d-1,  \quad 
\sum_{j=1}^{d-1} x_j \leq 1 \bigg\}.
 \]
We express the simplex using the particular local chart 
 $(m_1,\dots,m_{d-1},m_d) = (x_1,\dots, x_{d-1}, 1- \sum_{j=1}^{d-1} x_j)$.
Let $\Intx$ be the interior of $\hat{S}_d$ in $\R^{d-1}$; when we refer to the interior of simplex,  denoted by $\mathrm{Int}(S_d)$, we mean  the image of $\Intx$ under the above chart. Via the above local chart, a function $v$ defined on the simplex is equivalently written as a function $\hat{v}$ defined on 
 $\hat{S}_d$. Thus we say that $v\in \mathcal{C}^1(S_d)$ if $\hat{v}\in \mathcal{C}^1(\hat{S}_d)$, meaning that $\hat{v}\in \mathcal{C}^1(\mathrm{Int}(\hat{S}_d))$ with derivatives that extend continuously up to the boundary. In the interior of the simplex, derivatives are allowed only along directions $(\delta_j-\delta_i)_{i,j \in \ES}$, which are denoted as $\partial_{m_j-m_i}v(m)$; we define the vector
 \[
 D^i v (m) := (\partial_{m_j-m_i}v(m))_{j\in \ES}.
 \]
 Note that this is indeed a vector in $\R^{d-1}$ since $D^i_i v(m)=0$.
  The derivatives of $\hat{v}:\hat{S}_d\rightarrow \R$ are denoted by $D_x \hat{v}(x)=(\partial_{x_j}\hat{v}(x))_{j=1}^{d-1}$, with the obvious identity $\partial_{x_j}\hat{v}(x)= \partial_{m_j-m_d} v(m)$, $\forall j\in\ESd$, if $x$ represents $m$ in local chart.
 Let $\mathcal{C}^{1,1}(S_d)$ be the set of $v\in \mathcal{C}^1(S_d)$ whose derivative $D_x \hat{v}$ is Lipschitz-continuous, globally in $S_d$. 
 
 For $\bm{x}=(x_1,\dots,x_N)\in \ES^N$, denote the empirical measure 
 \be 
 \mu^N_{\bm{x}}:= \frac1N \sum_{k=1}^N \delta_{x_k}, \quad \mbox{ that is } 
 \mu^N_{\bm{x}}[i]= \frac1N \sum_{k=1}^N \mathbbm{1}_{\{x_k=i\}}, \quad i\in \ES.
 \ee 
 It takes values in the discretized simplex 
$S_d^N := S_d \cap \frac1N \mathbb{N}^d$. 
 We also denote the $N$-discretized derivative $D^{N,i}v = (D^{N,i}_j v)_{j\in\ES}$ of a function $v:S_d\rightarrow \R$, with 
 \be 
 D^{N,i}_j v(m):= N \bigg[ v \bigg( m+\frac1N(\delta_j-\delta_i)\bigg) -v(m) \bigg];
 \ee 
 while for a function $u: \ES^N\rightarrow \R$ we denote by $\Delta^k u \in \R^d$ the vector of differences, that is, 
 $\Delta^k u(\bm{x})[j]= u([\bm{x}^{-k},j]) -u(\bm{x})$, whereas, for $\bm{x}\in \ES^N$ and $j\in\ES$,  
 $[\bm{x}^{-k},j]$ denotes the vector in $\ES^N$ such that 
 $[\bm{x}^{-k},j]_l= \begin{cases}
 x_l &\mbox{ if } l\neq k \\
 j &\mbox{ if } l=k.
 \end{cases}$

\subsection{Assumptions}

Let $Q_{i,j} : [0,T] \times S_d \times A \rightarrow [0, +\infty)$ be the transition rate, $f^i:[0,T]\times A \times S_d \rightarrow \R$ the running cost and $g^i: S_d \rightarrow \R$ the terminal cost, where $i\in\ES$ represents the state. 
Let $F: [0,T]\times A^d \times S_d \rightarrow \R$ and 
$G:S_d\rightarrow \R$ be defined by 
\be 
\label{defFG}
F(t,a^1,\dots,a^d ,m) := \sum_{i\in\ES} m_i f^i(t,a^i,m) , \qquad 
G(m) = \sum_{i\in\ES} m_i g^i(m).
\ee 

We provide three sets of assumptions. 
The first is the weakest and gives convergence of the value functions, with a convergence rate. Note in particular that we don't assume convexity, neither in $a$, nor in $m$.

\begin{ass}
	\label{A}
	\begin{enumerate}
	\item[(A1)] The action space $(A,\mathbf{d})$ is a compact metric space.
	\item[(A2)] 	The transition rate  $Q_{i,j}$ is continuous on 
$[0,T]\times A \times S_d$ (thus uniformly continuous and bounded) and	
	Lipschitz-continuous in $(t,m)$, uniformly in $a$:
	\begin{equation}
	|Q_{i,j}(t,a,m) - Q_{i,j}(s,a,p)| \leq C(|t-s|+|m-p| ). 
	\end{equation} 
	\item[(A3)] The functions $F$ is continuous on 
	$[0,T]\times A^d \times S_d$ and
	\begin{align}
	|F(t,a,m)- F(s,a,m) | &\leq C(|t-s|+|m-p|  ), \\
	|G(m)-G(p) | &\leq C |m-p|.	
	\end{align}
	\end{enumerate}
	\end{ass} 

We denote then, for $z\in\R^d$ such that $z_i=0$, $a\in A$ and $m\in S_d$,  the pre-Hamiltonian 
\be 
\label{preH}
\mathcal{H}^i(t,a,m,z) := - \langle Q_{i,\bullet} (t,a,m) , z \rangle - f^i(t,a,m)
\ee
and the Hamiltonian 
\be 
\label{H}
H^i(t,m,z):= \max_{a\in A} \mathcal{H}^i(t,a,m,z).
\ee 
Note that $\sum_i m_i H^i$ is Lipschitz-continuous in $(t,m,z)$ if \ref{A} holds.

The second assumption is a linear-convex assumption, very common in control theory, which, together with existence of classical solution to the limiting problem, gives convergence of the optimal trajectories.

\begin{ass}
	\label{B}
	Assumption \ref{A} holds and, in addition:
	\begin{enumerate}
		\item[(B1)] $A = [0,M]^d$. 
		\item[(B2)] The transition rate is
		$Q_{i,j}(t,m,a) = a_j$. 
		\item[(B3)] The running cost $f$ is continuously differentiable in $A$, $\nabla_a f$ is Lipschitz-continuous with respect to $m$, and $f$ is uniformly convex in $A$, i.e. there exists $\lambda>0$ such that
		\be
		\label{conva}
		f^i(t,b,m) \geq f^i(t,a,m) + \langle \nabla_a f(t,a,m) , (b-a) \rangle + \lambda |b-a|^2.
		\ee 
	\end{enumerate}
\end{ass} 
Under this assumption, thanks to Proposition 1 in \cite{gomes}, there exists a unique maximizer of $\mathcal{H}$, which we denote 
by $a^*(t,i,m,z)$, and, further, $a^*$ is Lipschitz continuous with respect to $m$ and $z$, i.e. 
\be 
|a^*(t,i,m,z) - a^*(t,i,p,w)| \leq C (|m-p| +|z-w|) .
\ee 
We will consider  feedback controls $\alpha:[0,T] \times \ES \rightarrow A$ (or equivalently $\alpha:[0,T]\rightarrow A^d$), thus, when (\ref{B}1) holds, we denote $\alpha_{i,j}(t) := \alpha_j(t,i)$.

The last is a convexity assumption in the couple $(\alpha,m)$ that is needed to prove smoothness of the value function of the MFCP.
\begin{ass}
	\label{C}
	Assumption \ref{B} holds and, in addition:
	\begin{itemize}
		\item[(C1)] $F(\cdot,a,\cdot) \in \mathcal{C}^{1,1}([0,T]\times S_d)$, uniformly in  $a$, and $G\in \mathcal{C}^{1,1}(S_d)$.
		\item[(C2)] The function
		\[
		[0,T]\times [0,+\infty)^{d\times d} \times \Int \ni (t,w,m) 
		\rightarrow \sum_i m_i f^i \left(t, \Big(\frac{w_{i,j}}{m_i} \Big)_{j\neq i} , m \right) \in \R
		\]
		 is convex in $(w,m)$ and $G$ is convex in $m$.
	\end{itemize}
\end{ass} 
Assumption (\ref{C}2) may seem strange at  this stage, but will be clarified in Section \ref{sec3.2}.

\begin{rem}
The reader will notice that not all of the conditions of Assumption \ref{A} (Lipschitz-continuity in $(t,m)$) are necessary in all the statements in which \ref{A} is assumed. The same is true also for Assumption \ref{B}. We choose to make only three Assumptions  for the sake of definiteness. We stress again that the main differences among the three assumptions are that in \ref{A} nothing is  convex, while in \ref{B} we assume convexity in $a$ and in \ref{C} in $(a,m)$.

We remark that \ref{A} allows to treat also the case of directed graphs, in which some transitions are forbidden: if $E=\{e_1, \dots, e_d\}$ is the set of nodes and $E_+(e_i)$, for each $i\in\ES$, is the subset of $E\setminus\{e_i\}$ of nodes $e_j$ for which there exists a directed edge from $e_i$ to $e_j$, then the transition matrix $Q$ is required to satisfy $Q_{i,j}(t,a,m)=0$ whenever $e_j \notin E_+(e_i)$.

We also remark that we do not assume that $f$ splits in a function of $(i,a)$ plus a function of $(i,m)$, as it would be in the case of potential mean field games which are described in \S \ref{sec3.3}. 
\end{rem}

We conclude this part with a natural example for which \ref{B} and/or \ref{C} are satisfied.
\begin{exmp}
	As a natural running cost for which \eqref{conva} holds, we could take
	\be 
	\label{quadratic}
	f^i(t,\alpha,m) := \frac12 \sum_{j\neq i}  \alpha_{i,j}^2 + f^i_0(m).
	\ee 
	In this case, we get 
	\[
	H^i(t,m,z) = \sum_{j\neq i} 
	\left\{ 
	-\mathfrak{a}^*(-z_j)z_j -\frac12|\mathfrak{a}^*(-z_j)|^2
	\right\}  
	- f^i_0(m) ,
	\qquad a^*_j (t,i,z) = \mathfrak{a}^*(-z_j),
	\] 
	where, for $r\in\R$, $\mathfrak{a}^*(r) :=
	\begin{cases}
	0 \qquad&r\leq 0 \\
	r &0\leq r\leq M \\
	M &r\geq M
	\end{cases}
	$. 	
	
	In fact, we remark that our Assumptions \ref{B} and \ref{C} are more general since they allow for a running cost which does not split as in \eqref{quadratic}. 
	If $f$ is as in \eqref{quadratic}, then (\ref{C}2) is satisfied if the function $m\rightarrow \sum_i m_i f^i_0(m)$ is convex; indeed, one can easily verify that the function 
	$(w,m)\rightarrow \sum_i  \sum_{j\neq i} \frac{w_{i,j}^2}{m_i}$ is convex in $(w,m)$.
\end{exmp}

\subsection{$N$-agent optimization}

Consider $N$ players, $\bm{X}=(X^1,\dots,X^N)$, such that $X^i_t\in \ES$, evolving in continuous time over a finite horizon $T$. 
Agents can choose controls $\bm{\beta} =(\beta_1,\dots,\beta_N)$ in feedback form, i.e. any $\beta_k$ is a measurable function of time and state of all players: $\beta_k: [0,T]\times \ES^N \rightarrow A$. 
The dynamics is given as a Markov chain such that 
\be 
\label{tranXN}
\P(X^k_{t+h}=j |\bm{X}_t=\bm{x} )= Q_{x_k,j}(t,\beta_k(t,\bm{x}), \mu^N_{\bm{x}}) h +o(h),
\ee
for $j\neq x_k$, as $h\rightarrow 0^+$.
Agents are cooperative and aim at minimizing the common cost
\be 
\label{costN}
J^N(\bm{\beta}) := \frac1N \sum_{k=1}^N \E\left[ \int_0^T   f(t,X^k_t, \beta_k (t, \bm{X}_t),\mu^N_t)
 dt + g(X^k_T,\mu^N_T)  \right].
\ee
The cost coefficients $f$ and $g$ depend on the empirical measure 
\be
\mu^N_t= \frac1N \sum_{k=1}^N \delta_{X^i_t} \qquad \in S_d^N.
\ee
We denote $\mu^N=(\mu^N_1,\dots,\mu^N_d)$, that is,  
$\mu^N_{i,t}=\frac1N \sum_{k=1}^N \mathbbm{1}_{\{X^k_t=i\}}$.

The initial conditions $(X^1_0,\dots,X^N_0)$ are assumed to be i.i.d with $\Law(X^1_0)=m_0$.
This can be seen as a single optimization problem for the process 
$\bm{X}$, governed by the generator
\be 
\label{genN}
\mathcal{L}^{N,\bm{\beta}}_t \phi(\bm{x}) = \sum_{k=1}^N \sum_{j\neq x_k} 
Q_{x_k,j} (t,\beta_k(t,\bm{x}), \mu^N_{\bm{x}} ) [\phi(\bm{x}^{-k},j) - \phi(\bm{x})], 
\ee  
for any $\varphi: \ES^N\rightarrow \R$.
The value function $v^N$ of this control problem solves the HJB equation 
\[
-\frac{d}{dt} v^N(t,\bm{x}) + 
\sup_{\bm{b}\in A^N} \bigg\{-\mathcal{L}^{N,\bm{b}}_t v^N(t,\bm{x}) - \frac1N \sum_{k=1}^N f(t,x_k, b_k, \mu^N_{\bm{x}}) \bigg\} =0, 
\]
which, by definition of the Hamiltonian in \eqref{H}, gives 
\be 
\label{NODE}
\begin{split}
&-\frac{d}{dt} v^N(t,\bm{x}) + \frac1N \sum_{k=1}^N H^{x_k}
(t, \mu^N_{\bm{x}}, N \Delta^k v^N(t,\bm{x}))) =0, \\
&v^N(T,\bm{x}) = \frac1N \sum_{k=1}^N g(x_k, \mu^N_{\bm{x}}).
\end{split}
\ee 
This is a system of ODEs indexed by $\bm{x}\in\ES^N$.

\begin{prop}
	Under Assumption \ref{A}, there exists a unique solution $v^N$ to \eqref{NODE}, which is $\mathcal{C}^1$ in time. It is the value function of the control problem \eqref{costN}-\eqref{genN} and there exists an optimal feedback control.
\end{prop}

\begin{proof}
This is a standard verification theorem (see e.g. 
\cite[Thm. III.8.1]{fleso}) and existence of solution is given by the Lipschitz continuity of the Hamiltonian in \eqref{NODE}, which follows by the continuity of the coefficients and compactness of $A$. These properties also yield existence of a maximizer in \eqref{preH} and hence existence of an optimal feedback.
\end{proof}

\begin{rem} 
	The optimal control is not unique and there might exist non-exchangeable optimizers. We recall that a vector of stochastic processes is said to be exchangeable if its joint law is invariant under permutations.
	
	 Under Assumption \ref{B} the optimal control is unique. 
	\end{rem}

\begin{rem}
	For problem \eqref{costN}-\eqref{genN}, the choice of controls in Markovian feedback form is made for convenience only and is the most natural setup for this type of control problems. We could consider more general open-loop controls: in this case, the strategy vector $(\pi^1_t, \dots, \pi^N_t)_{t\in[0,T]}$ is a vector of predictable $A$-valued processes and the dynamics of the state process $\bm{X}$ can be defined as the solution of the controlled martingale problem related to the generator  \eqref{genN}, in which $\beta_k$ is replaced by the stochastic process $\pi^k$. Notably, the value function of this more general control problem still solves Equation \eqref{NODE}, which admits a unique solution, and thus these two control problems are equivalent. 
\end{rem}


\subsubsection{Mean field formulation}
We give another formulation of the $N$-agent optimal control, by restricting the class of admissible controls. Nevertheless, we will prove that the two formulations are equivalent, in the sense that the value function is the same and thus the infimum of the cost is the same. 

Assume then that the control is the same for any player and is given by a feedback (measurable) Markovian function 
$\alpha_N : [0,T]\times \ES \times S_d^N \rightarrow A$: we make the mean field assumption for which the control depends on the private state and on the state of  other players only through the empirical measure 
$\mu^N_t$ of the entire system. 
Namely, we assume that, for any $k \in \NN$,
\be 
\label{betaalpha}
\beta_k(t,\bm{x}) = \alpha_N(t, x_k, \mu^N_{\bm{x}}).
\ee 
 Thus we have 
\be 
\label{tranNmu}
\P(X^k_{t+h}=j | X^k_t=i, \mu^N_t =m)= 
Q_{i,j}(t,\alpha_N(t,i, m ), m) h +o(h),
\ee
as $h\rightarrow 0^+$. 

The aim of the players is 
 to choose $\alpha^N$ in order to minimize the cost in \eqref{costN}, which, assuming now  \eqref{betaalpha}, can be  rewritten as
\begin{align*} 
J^N(\alpha_N) &:= \frac1N \sum_{k=1}^N \E\left[ \int_0^T  f(t,X^k_t, \alpha_N (t, X^k_t, \mu^N_t), \mu^N_t ) dt + g(X^k_T,\mu^N_T)  \right] \\
&= \frac1N \sum_{k=1}^N \E\left[  \int_0^T 
\sum_{i\in\ES} \mathbbm{1}_{\{X^k_t=i\}}
  f(t, i, \alpha_N (t, i,\mu^N_t) , \mu^N_t)  dt + 
  \sum_{i\in\ES} \mathbbm{1}_{\{X^k_T=i\}} g(i,\mu^N_T) \right] ,
  \end{align*}
  and hence 
  \be 
  \label{costNmu}
   J^N(\alpha_N)=\E\left[  \int_0^T 
  \sum_{i\in\ES} \mu^N_{i,t}
  f^i(t,  \alpha_N (t, i,\mu^N_t) , \mu^N_t)  dt + 
  \sum_{i\in\ES} \mu^N_{i,T} g^i(\mu^N_T) \right].
\ee
Therefore the $N$-agent mean field control problem can be seen as a single optimization problem for the empirical measure, which is a time-inhomogeneous Markov chain on $S_d^N$ such that 
\be 
\label{ratemu}
\P\left(\left.\mu^N_{t+h}=m+\frac1N (\delta_j-\delta_i) \right| \mu^N_t =m\right)= N m_i Q_{i,j}(t, \alpha_N(t,i,m), m) h +o(h)
\ee
for any $m\in S_d^N$ and $i\neq j \in \ES$.
The control (in feedback form) is now the vector valued measurable function 
$\alpha^N(t,\cdot, m) \in A^d$.
 We remark that the dynamics remains in $S^N_d$ because $m+\frac1N (\delta_j-\delta_i)$ can be outside $S^N_d$ only if $m_i=0$, but in such case the transition rate is zero. 
Clearly, we assume that $\mu_0=  \frac1N \sum_{k=1}^N \delta_{X^i_0}$.
The generator of this Markov chain is hence given, for $v:S^N_d\rightarrow \R$, by
\be 
\label{genNmu}
\mathcal{L}^{N,\alpha_N}_t v(m) = N \sum_{i,j\in\ES} m_i Q_{i,j}(t,\alpha_N(t,i,m),m) \left[ v\left(m+\frac1N (\delta_j-\delta_i)\right) -v(m)\right].
\ee

The HJB equation for the value function of this problem is then
\[
-\frac{d}{dt} V^N(t,m)+\max_{a\in A^d}\bigg\{
-\mathcal{L}^{N,a} V^N(t,m) - \sum_{i\in \ES} m_i f^i(t,a^i,m)
\bigg\} =0,
\]
which rewrites as
\be
\label{HJBN}
\begin{split}
&-\frac{d}{dt}V^N + \sum_{i\in \ES} m_i H^i(t, m, D^{N,i}V^N(t,m)) =0\\
&V(T,m)=\sum_{i\in \ES} m_i g^i(m),
\end{split}
\ee
that is an ODE indexed by $m\in S^N_d$.

\begin{prop}
\label{prop5}
Under Assumption \ref{A}, the HJB equation \eqref{HJBN} admits a unique solution $V^N$ which is $\mathcal{C}^1$ in time; $V^N$ is the value function of the control problem \eqref{costNmu}-\eqref{ratemu}, defined as the infimum over feedback controls $\alpha_N:[0,T] \times S^N_d \rightarrow A^d$, and there exists an optimal feedback.  $V^N$ satisfies the Lipschitz property
\be 
\label{eq:23}
|V^N (t, m) -V^N(s,p) | \leq C (|t-s|+ |m-p|)
\ee
for a constant $C$ independent of $N$. 

Moreover, the value function $v^N$ of \eqref{costN}-\eqref{genN} satisfies 
\be
\label{vVN} 
v^N(t, \bm{x}) := V^N(t, \mu^N_{\bm{x}})
\ee 
 and thus the control problems
 \eqref{costN}-\eqref{genN} and \eqref{costNmu}-\eqref{ratemu} are equivalent, i.e. 
 \be 
 \label{JBA}
 \inf_{\bm{\beta}} J^N(\beta) = \inf_{\alpha_N} J^N(\alpha_N).
 \ee  
\end{prop}

\begin{proof} 
	The first claim follows by the Lipschitz continuity of $H$ and a standard verification theorem, an optimal feedback being a measurable function that attains the maximum in \eqref{preH}, which exists by compactness of $A$ and continuity of the coefficients.   
	The Lipschitz continuity of $V^N$ is proved in Section \ref{Sec:4.1}, Lemma \ref{lemlip}. If $w^N$ is the function defined by the r.h.s. of \eqref{vVN}, then 
	\begin{align*} 
		&-\frac{d}{dt} w^N(t,\bm{x}) + \frac1N \sum_{k=1}^N H^{x_k}
		(t, \mu^N_{\bm{x}}, N  (w^N(t,j,\bm{x}^{-k}) -w^N(t,\bm{x})  )_{j\in\ES})  \\
		&= -\frac{d}{dt} V^N(t, \mu^N_{\bm{x}}) + \frac1N \sum_{k=1}^N \sum_{i\in \ES} \mathbbm{1}_{\{x_k=i\}}H^{i}
		(t, \mu^N_{\bm{x}}, N (V^N(t, \mu^N_{\bm{x}} +\frac1N(\delta_j-\delta_i)  )  -V^N(t, \mu^N_{\bm{x}})  )_{j\in\ES} )\\
		&= -\frac{d}{dt} V^N(t, \mu^N_{\bm{x}}) +  \sum_{i\in \ES} 
		\mu^N_{\bm{x}}[i]
		H^{i}
		(t, \mu^N_{\bm{x}}, D^{N,i}V^N(t, \mu^N_{\bm{x}}) )  =0
	\end{align*}
	because $V^N$ solves \eqref{HJBN}. Hence, by uniqueness of the solution to \eqref{NODE}, we have $w^N = v^N$ and thus \eqref{vVN} is satisfied.	
\end{proof}

\begin{rem}
\label{remeq}
	We could restrict the class of admissible controls to the set of 
	$\alpha_N:[0,T] \rightarrow A^d$ that are deterministic functions only of time and of the private state $i\in \ES$. This control problem is equivalent to what we consider here because its value function $W^N$ is clearly Lipschitz in time, and hence absolutely continuous, and then, by the dynamic programming principle, it is easy to show that $W^N$ solves \eqref{HJBN} at any point of differentiability. Thus, by uniqueness of the solution to \eqref{HJBN} (defined in the sense of Caratheodory in the class of absolutely continuous functions)
	 we get $V^N = W^N$, which means that the costs have the same infimum. 
	 
	 We could further restrict the class of admissible controls, for instance,  to the set of piecewise constant (deterministic) functions of time. This setting is the one considered in \cite{kol}.
\end{rem}

If assumption \ref{B} holds, then the  optimal control is unique, since the maximizer in \eqref{preH} is unique (see e.g. \cite[Thm. 5]{cecfis}).
The control is the transition rate $\alpha_N(t,i,m) \in [0,M]^d$, which we denote as a transition matrix $\alpha_N = (\alpha_N^{i,j})_{i,j\in\ES} \in [0,M]^{d\times d}$.
In such case, \eqref{tranNmu} becomes simply
\be 
\label{tranopt}
\P(X^k_{t+h}=j | X^k_t=i, \mu^N_t =m)= 
\alpha_N^{i,j}(t, m ) h +o(h), 
\ee
and the dynamics of $\mu^N$ is then given by
\be 
\label{muNopt}
\P\left(\left.\mu^N_{t+h}=m+\frac1N (\delta_j-\delta_i) \right| \mu^N_t =m\right)= N m_i \alpha_N^{i,j}(t,m)h +o(h).
\ee
Note that the values of $\alpha^{i,i}$ never enter in the dynamics.  Then it follows:
\begin{prop}
Under Assumption \ref{B}, there exists a unique optimal control. It is  given by 
\be 
\label{optalN}
\alpha_{N,*}^{i,j}(t,m) := a^*_j (t,i,m, D^{N,i}V^N(t,m) ).
\ee	 
\end{prop}

\subsection{Mean field control problem}

In the limit, there is a single player which evolves according to 
\be 
\label{limX}
\P(X_{t+h}=j | X_t=i)= Q_{i,j}(t, \alpha^i(t), \Law(X_t)) h +o(h)
\ee
and $\Law (X_0)=m_0$. The control is here a deterministic measurable function 
$\alpha:[0,T] \rightarrow A^d$, or equivalently $\alpha:[0,T]\times \ES \rightarrow A$, which is indeed a feedback function of the state $X_t$, denoted by $i\in \ES$. As a particular case, this set includes the functions $[0,T] \times S_d \ni (t,\Law(X_t))\rightarrow \tilde{\alpha} (t,\Law(X_t)) \in A^d$, since $t\rightarrow \Law(X_t)$ is a deterministic function of time. 
The reference player aims at minimizing the cost
\be
\label{costmf} 
J(\alpha) :=  \E\left[ \int_0^T   f(t, X_t, \alpha (t, X_t),  \Law(X_t) )  dt + g(X_T,\Law(X_T))  \right].
\ee
The problem can be recasted into a deterministic control problem for the dynamics of the law of $X$, thanks to the fact that we consider Markovian feedback controls only. Indeed, denoting $\mu_t = \Law(X_t)$, i.e. $\mu_t^i = \P(X_t=i)$, the cost is written as 
\be 
\label{costmu}
J(\alpha) = \int_0^T \sum_{i\in \ES}  f^i( t,\alpha^i (t),\mu_t) \mu^i_{t} dt + \sum_{i\in \ES}g^i(\mu_T)\mu^i_{T} ,
\ee
where $\mu$ solves the ODE, indexed by $i\in\ES$, 
\footnote{
For the deterministic control problem \eqref{costmu}-\eqref{mu} the control $\alpha:[0,T]\rightarrow A^d$ is indeed in open-loop form, as it is just a measurable function of time. 
}

\be 
\label{mu}
\begin{split}
&\frac{d}{dt}\mu^i_t = \sum_{j\in\ES} \Big( \mu^j_t Q_{j,i}(t,\alpha^j(t), \mu_t)  - \mu^i_t Q_{i,j}(t,\alpha^i(t),\mu_t)\Big)\\
& \mu_0 =m_0.
\end{split}
\ee
The HJB equation for the value function $V$ of this problem is then
\[
-\partial_t V(t,m)+\max_{a\in A^d}\left\{
-\sum_{i,j\in\ES} m_i Q_{i,j}(t, a^i,m) \partial_{m_j-m_i} V(t,m)   - \sum_{i\in \ES} m_i f^i(t,a^i,m)
\right\}=0,
\]
which rewrites as 
\be
\label{HJB}
\begin{split}
& -\partial_t V(t,m) +  \sum_{i\in \ES} m_i H^i(t, m, D^{i}V(t,m)) =0,\\
&V(T,m)=\sum_{i\in \ES} m_i g^i(m),
\end{split}
\ee
in which we recall that 
$
[D^{i} V(t,m)]_j:= \partial_{m_j-m_i}V(t,m).
$

The peculiarity of this first-order equation is that it is stated in the simplex, which is a bounded domain, but there are no boundary conditions. This is explained by the fact that the simplex is invariant for the dynamics \eqref{mu}, and so is its interior. The first order HJB equation has no classical solutions in general,  and for this reason viscosity solutions were introduced. These are defined properly in the next section, by viewing the simplex as a subset of $\R^{d-1}$ instead of $\R^d$. Viscosity solutions can be defined either on $S_d$ or on $\Int$, depending on the boundary regularity of the test functions involved. One problem with defining viscosity solutions on $S_d$ is that it is not clear that a classical solution is a viscosity solution on $S_d$; however, this definition is the one we will use to prove convergence of $V^N$ to $V$. Exploiting the fact that $\Int$ is invariant for the state dynamics, which results in a property on the subdifferential of $H$, it is possible to show uniqueness of viscosity solutions on $\Int$; see the comparison principle below (Theorem \ref{thmcomparison}).  

If \ref{B} holds, we denote as above $\alpha_{i,j}(t)=\alpha_j(t,i)$, which is the transition rate matrix.
In the next section, we prove the following:

\begin{thm}
\label{thm:mainmfc}
	Let $V$ be the value function of the deterministic control problem \eqref{costmu}-\eqref{mu}:
	\begin{enumerate}
	\item if Assumption \ref{A} holds,  then $V$ is the unique viscosity solution of \eqref{HJB} on $S_d$, and $V$ is Lipschitz-continuous in $(t,m)$; 
	\item if Assumption \ref{A} holds, then $V$ is the unique viscosity solution of \eqref{HJB} on $\Int$, and, if \ref{B} holds, there exists an optimal control; 
	\item if Assumption \ref{C} holds, then $V\in \mathcal{C}^{1,1}([0,T]\times S_d)$ and is the unique classical solution of \eqref{HJB};
	\item if $V\in \mathcal{C}^{1,1}(S_d)$ and \ref{B} holds, then 
	  the control  given by the feedback 
	\be 
	\label{alphaopt}
	\alpha_{*}^{i,j}(t,m) := a^*_j (t,i,m, D^{i}V(t,m) )
	\ee  
	is the unique optimal control, in the sense that any optimal control $\alpha:[0,T]\rightarrow [0,M]^{d\times d}$, with related optimal process $\mu$ is such that $\alpha(t) = \alpha_*(t,\mu_t)$ for $dt$-a.e. $t\in [0,T]$.  
	\end{enumerate}
\end{thm}


\subsection{Convergence results} 

We state here the  results about the convergence, as $N\rightarrow \infty$, of the value function $V^N$ of the $N$-agent optimization \eqref{costNmu}-\eqref{ratemu} to the value function $V$ of the mean field control problem 
\eqref{costmu}-\eqref{mu}, with a convergence rate of order $1/\sqrt{N}$. We recall that $V^N$ is the classical solution to ODE \eqref{HJBN}, while $V$ is the viscosity solution to PDE \eqref{HJB}. 
The following is our main result:

\begin{thm}
\label{thmconvV}
Under Assumption \ref{A},
\be 
\label{convV}
\max_{t\in [0,T], m\in S_d^N} 
	\left|V^N(t,m) - V(t,m)\right| \leq \frac{C}{\sqrt{N}}.
\ee
\end{thm}

The theorem is proved in Section \ref{Sec:4.1}. As announced in the Introduction, we exploit the characterization of $V$ as the viscosity solution to \eqref{HJB} in order to prove the result. In fact, ODE \eqref{HJBN} can be seen as a finite difference scheme for the PDE \eqref{HJB}, even if time is still continuous. Indeed the argument $D^{N,i}V$ of the Hamiltonian in \eqref{HJBN} converges, at least formally, to $D^i V$ appearing in \eqref{HJB}, as 
\be 
\lim_{N\rightarrow \infty} N \bigg[ V \bigg( m+\frac1N(\delta_j-\delta_i)\bigg) -V(m) \bigg] = \partial_{m_j-m_i} V(m).
\ee 


This result also permits to construct quasi-optimal controls for the $N$-agent optimization, starting from quasi-optimal controls for the MFCP, with an explicit rate of approximation.

\begin{thm}
\label{thmappr}
Assume \ref{A} and fix $\epsilon>0$ and $N\in \mathbb{N}$. Let $\alpha:[0,T] \rightarrow A^d$ be an $\epsilon$-optimal control for the MFCP. Then 
\be 
\label{appr}
J^N(\alpha) \leq \inf_{\alpha_N} J^N(\alpha_N) + \frac{C}{\sqrt{N}} + \epsilon.
\ee 
\end{thm} 
This is also proved in Section \ref{Sec:4.1}.
Here, $J^N(\alpha)$ is understood as applying the control $\alpha_N(t,m) = \alpha(t)$, which is independent of $m$. 
Recall that the infimum over controls $\alpha_N$ is the same as the infimum over controls $\bm{\beta}$, depending on states of all the players, by  \eqref{JBA}, and is also equal to the infimum over controls not depending on $m$ (like the $\alpha$ considered), by Remark \ref{remeq}.

\begin{rem}
\label{remk}
In \cite{kol}, Kolokoltsov proved a result similar to Theorem \ref{thmconvV}, but assuming in addition that $F$, $G$ and $Q$ and $\mathcal{C}^{1,1}$ w.r.t. $m$ (similarly to Assumption (\ref{C}1)).
He analyzes a mean field $N$-optimization like in \eqref{costNmu}-\eqref{genNmu}, but allowing for controls that are piecewise constant functions of time only, and are the same controls he considers in the limiting deterministic control problem \eqref{costmu}-\eqref{mu}. However, we explained in Remark \ref{remeq} that considering this smaller class is not restrictive, as the value of the $N$-agent optimization is the same as what we treat here, i.e. over controls that might depend also on $m$. Then, by applying the convergence of the generator \eqref{genNmu} to the limiting dynamics \eqref{mu}, he shows convergence of the value functions with a stronger convergence rate (Theorem 2 therein):
\be 
\label{convK}
\max_{t\in [0,T], m\in S_d^N} 
	\left|V^N(t,m) - V(t,m)\right| \leq \frac{C}{N}.
\ee

As a matter of fact, from his method of the proof (basically, the same set of controls is considered for the prelimit and the limit optimization problems), it is also possible to derive estimate \eqref{convV}. Indeed, by applying standard arguments in propagation of chaos, we can get a convergence rate of order $1/\sqrt{N}$, assuming that the  costs and the transition rate are just Lipschitz-continuous w.r.t. $m$, and not in $\mathcal{C}^{1,1}$.
Therefore, what we propose in this paper is a new method for proving the convergence in \eqref{convV}, based on the theory of viscosity solutions, which we believe can be of interest. 
\end{rem}

In case $V$ is smooth, the optimal control of the MFCP is unique and then, if \ref{B} holds, we are also able to establish a \emph{propagation of chaos} result, that is,  we prove convergence of the optimal trajectory of the $N$-agent optimization to the unique optimal trajectory of the MFCP, with a suitable convergence rate. 

Denote then by $\alpha^N$ the unique optimal feedback control for the $N$-agent optimization defined by \eqref{optalN}, and by $\mu^N$ the corresponding optimal process satisfying \eqref{muNopt}. Also, let $\alpha_*$ be the unique optimal feedback control for the MFCP defined by \eqref{alphaopt} and $\mu$ the corresponding optimal trajectory given by \eqref{mu}. We stress that $\alpha_N$ and $\alpha_*$ are functions of $t$ and $m$. 


\begin{thm}[Part I]
\label{thmchaos}
Under Assumption \ref{B}, if $V\in\mathcal{C}^{1,1}([0,T]\times S_d)$ then 
\be 
\label{chaosmu}
\E \left[ \sup_{t\in [0,T]} |\mu^N_t -\mu_t |  \right] \leq \frac{C}{N^{1/9}} .
\ee 
\end{thm} 

The propagation of chaos result can be stated also for the vector of processes $\bm{X}$ related to the optimal control $\alpha_N$ and optimal empirical measure $\mu^N$, that is, $\bm{X}$ is given by \eqref{tranXN} assuming \eqref{betaalpha}. For $N$ fixed, 
denote by $\tilde{\bm{X}}$ the i.i.d process (given by \eqref{tranXN}) in which all players choose the same local control $\alpha(t,i) :=\alpha_*(t,i,\mu_t)$ depending only on the private state, i.e. 
$\beta_k(t, \bm{x})= \alpha(t, x_k)$. The propagation of chaos consists in proving convergence of $\bm{X}$ to the i.i.d. process $\tilde{\bm{X}}$.

{\begin{customthm}{\ref{thmchaos}}[Part II]
\emph{Under Assumption \ref{B} and (\ref{C}1), if $V\in\mathcal{C}^{1,1}([0,T]\times S_d)$ then }
\be
\label{chaosXtilde} 
\E \left[ \sup_{t\in [0,T]} |X^i_t -\tilde{X}^i_t |  \right] \leq \frac{C}{\sqrt{N}} .
\ee 
\end{customthm}

\section{Mean Field Control Problem}
\label{sec3}

The aim here is to examine in detail the mean field control problem \eqref{costmu}-\eqref{mu} in order to prove Theorem \ref{thm:mainmfc}. 
We first rewrite the state dynamics and the cost in terms of the local chart $x\in \R^{d-1}$ (recall that we set $(m_1,\dots,m_{d-1},m_d)=(x_1,\dots, x_{d-1}, 1- \sum_{j=1}^{d-1} x_j)$),  so that we are allowed to apply standard results about deterministic control probles on Euclidean spaces; see e.g. \cite{BCD, cannarsa, fleso}. 

For $x\in \hat{S}_d$, let us denote $x^{-d}= 1-\sum_{j=1}^{d-1} x_j$ and $\check{x}= (x,x^{-d}) \in S_d$. We simply replace $m_d$ by $x^{-d}$: dynamics \eqref{mu} becomes, for $i=1,\dots, d-1$,
\be 
\label{dynx}
\begin{split} 
\frac{d}{dt} x^i_t &= \sum_{j=1}^{d-1} \Big( x^j_t Q_{j,i}(t,\alpha^j(t), \check{x}_t ) - x^i_t Q_{i,j}(t, \alpha^i(t), \check{x}_t) \Big) \\
& \quad +  x^{-d}_t Q_{d,i}(t,\alpha^d(t), \check{x}_t ) - x^i_t Q_{i,d}(t, \alpha^i(t), \check{x}_t) ,
\end{split}
\ee
while the cost \eqref{costmu} is written as 
\be 
\label{costx}
\hat{J}(\alpha) := \int_0^T \Big( \sum_{j=1}^{d-1} x^j_t f^j(t,\alpha^j(t), \check{x}_t) + x^{-d}_t f^d(t,\alpha^d(t), \check{x}_t) \Big) dt + \sum_{j=1}^{d-1} x^j_T g^j(\check{x}_T) 
+ x^{-d}_T g^d(\check{x}_T) .
\ee 
It is clear that the value function $\hat{V}$ of \eqref{dynx}-\eqref{costx}, defined on $\hat{S}_d$,  is equal to the value function $V$ of \eqref{costmu}-\eqref{mu}, defined on $S_d$, by setting
$\hat{V}(t,x)= V(t, \check{x})$. The HJB equation for $\hat{V}$ in $\hat{S}_d$ is then 
\be
\label{HJBx}
\begin{split}
&-\partial_t \hat{V} + \sum_{i=i}^{d-1} x_j \hat{H}^j(t, x, D_x \hat{V}) + x^{-d} \hat{H}^d(t, x, D_x \hat{V} )=0 \\
&\hat{V}(T,x)= \sum_{i=1}^{d-1} x_j g^j( \check{x}) + x^{-d} g^d(x), 
\end{split} 
\ee
where the modified Hamiltonian is defined, for $z\in \R^{d-1}$, by
\be 
\label{Hx}
\hat{H}^i(t,x,w)= H^i(t,\check{x}, w_1-w_i, \dots, w_{d-1}-w_i, -w_i), \qquad 
\hat{H}^d(t, x, w) = H^d(t,\check{x}, w,0).
\ee 

We will use several times the fact that $\mathrm{Int}(S_d)$ is invariant for the dynamics \eqref{mu}, or equivalently that $\mathrm{Int}(\hat{S}_d)$ is invariant for \eqref{dynx}. 
Indeed, by assumption \ref{A} follows that $Q$ is bounded (as it is continuous on a compact set): let us set
\be 
M:= \max_{i,j\in \ES, t\in [0,T], a\in A, m\in S_d} Q_{i,j}(t,a,m)
\ee 
and recall that $Q_{i,j}\geq 0$. Then \eqref{mu} gives
\[
\frac{d}{dt} \mu^i_t \geq 0- M(d-1) \mu^i_t ,
\]
which, by  Gronwall's inequality, provides
\be 
\label{invInt}
\mu_t^i \geq \mu_0^i e^{- T M(d-1)},
\ee
meaning that $\mu_t\in \mathrm{Int}(S_d)$ if $\mu_0 \in \mathrm{Int}(S_d)$. 
 
\subsection{Viscosity solution} 
 
 The value function in not $\mathcal{C}^1$ in general and thus does not solve \eqref{HJB} in the classical sense, unless the convexity assumption \ref{C} holds; see the next Subsection. 
Hence we  present the two definitions of viscosity solutions we make use of: one in $S_d$ and one in $\mathrm{Int}(S_d)$. 
The notion of viscosity solution is the usual one, but we prefer to state them to avoid confusions, because the use of test functions defined on  a closed set is not standard. 
By our definition of $\mathcal{C}^1(S_d)$, viscosity solutions can be defined in two equivalent ways, since it is equivalent to define functions on $S_d$ or on $\hat{S}_d$, by the relation $\hat{v}(x)=v(\check{x})$. Thus we prefer to state only the definitions of solutions to \eqref{HJB} on $S_d$ or $\Int$; the definitions of solutions to \eqref{HJBx} on $\hat{S}_d$ or $\mathrm{Int}(\hat{S}_d)$ being equivalent.   

\begin{defn}
\label{defvisco}
A function $v\in \mathcal{C}([0,T)\times S_d) $  is said to be:
\begin{itemize}
\item[(i)] a \emph{viscosity subsolution} of \eqref{HJB} on $S_d$ 
(resp. on $\Int$) if, for any 
$\phi\in \mathcal{C}^1([0,T)\times S_d)$ (resp. in $ \mathcal{C}^1([0,T)\times \Int)$ ), 
\be 
\label{viscosub}
-\partial_t \phi(\bar{t},\bar{m}) + 
\sum_{i\in\ES} \bar{m}_i H^i(\bar{t}, \bar{m}, D^i \phi(\bar{t},\bar{m}) ) \leq 0,
\ee
at every  
$(\bar{t},\bar{m})\in [0,T)\times S_d$ (resp. in $[0,T)\times \Int)$ which is a local maximum of $v-\phi$ on 
$[0,T)\times S_d$ (resp. on $[0,T)\times \Int$);
\item[(ii)] a \emph{viscosity supersolution} of \eqref{HJB} on  $S_d$ 
(resp. on $\Int$) if, for any 
$\phi\in \mathcal{C}^1([0,T)\times S_d)$ (resp. in $ \mathcal{C}^1([0,T)\times \Int)$ ), 
\be 
\label{viscosuper}
-\partial_t \phi(\bar{t},\bar{m}) + 
\sum_{i\in\ES} \bar{m}_i H^i(\bar{t}, \bar{m}, D^i \phi(\bar{t},\bar{m}) ) \geq 0,
\ee
at every  
$(\bar{t},\bar{m})\in [0,T)\times S_d$ (resp. in $[0,T)\times \Int)$ which is a local minimum of $v-\phi$ on 
$[0,T)\times S_d$ (resp. on $[0,T)\times \Int$);
\item[(iii)]a \emph{viscosity solution} of \eqref{HJB} on $ S_d$ (resp on $\Int$) if it is both  a viscosity subsolution and a viscosity supersolution of \eqref{HJB} on $ S_d$ (resp on $\Int$).
\end{itemize}
\end{defn}

Let us remark that, in the above definition, by solutions to \eqref{HJB} we clearly mean solutions to the first line of \eqref{HJB}. As an example  of test functions in $\mathcal{C}^1([0,T)\times S_d)$, we may consider continuously differentiable functions defined on an open subset of $\R^{d-1}$ containing $\hat{S}_{d}$. With this remark, it  is straightforward to obtain the following result, which proves point (1) in Theorem \ref{thm:mainmfc}:

\begin{prop}
\label{propvisco1}
Under Assumption \ref{A}, the value function $V$ of the MFCP is the unique viscosity solution of \eqref{HJB} on $S_d$ (in $\mathcal{C}([0,T]\times S_d)$ satisfying the terminal condition). Moreover, $V$ is globally Lipschitz-continuous in $[0,T] \times S_d$. 
\end{prop} 

\begin{proof}
The  Lipschitz-continuity and the  viscosity solution property of $V$ follow from standard results in deterministic control theory, see Thm. 7.4.10 and 7.4.14 of \cite{cannarsa}, by considering the dynamics in $\R^{d-1}$ and using the fact that $\hat{S}_{d-1}$ is invariant for the dynamics \eqref{dynx}. Uniqueness of viscosity solutions on $S_d$ follows by the usual proof of uniqueness, see for instance \cite[Thm. II.9.1]{fleso}, by observing that --if the minimizers are on boundary points-- we can use the fact that the test functions constructed in the proof are quadratic and thus defined in the whole $\R^{d-1}$, in particular then  belonging to $\mathcal{C}^1([0,T]\times S_d)$. We believe that there is no need to rewrite the proof here.  
\end{proof} 

This notion of viscosity solution on $S_d$ will be used in the proof of the convergence result Theorem \ref{thmconvV}, see Subsection \ref{Sec:4.1}. Actually, 
uniqueness of viscosity solutions on $S_d$ can also be derived as a consequence of the proof therein. 
The problem with the definition on the closed set $S_d$ is that it is not clear whether a classical solution is a viscosity solution on $S_d$. Indeed, if the value function $V$ is smooth, then for sure $V$ is a viscosity solution on $\Int$, but if a maximizer of $V-\psi$ lies on the boundary of $S_d$ then it is not clear a priori  that \eqref{viscosub} holds. We could prove this fact, but we prefer instead to show uniqueness of viscosity solution on $\Int$, which implies in particular that $V\in \mathcal{C}^1([0,T]\times S_d)$ is a viscosity solution on $S_d$, but is more general and we believe can have an interest on itself. The following result then proves points (2) and (4) in Theorem \ref{thm:mainmfc}.

\begin{prop}
Under Assumption \ref{A}, $V$ is the unique viscosity solution of \eqref{HJB} on $\Int$, in $\mathcal{C}([0,T]\times S_d)$ satisfying the terminal condition. Moreover, if \ref{B} holds:
\begin{enumerate}
\item 
there exists an optimal control for the MFCP; 
\item if $V\in \mathcal{C}^{1,1}([0,T]\times S_d)$, then 
	  the control  given by the feedback 
	  \eqref{alphaopt}
	is the unique optimal control, in the sense that any optimal control $\alpha:[0,T]\rightarrow [0,M]^{d\times d}$, with related optimal process $\mu$ is such that $\alpha(t) = \alpha_*(t,\mu_t)$ for $dt$-a.e. $t\in [0,T]$.  
\end{enumerate}
\end{prop} 

\begin{proof}
Existence of an optimal control follows from \cite[Thm 7.4.5]{cannarsa}, using the convexity of the cost in $a$. The viscosity solution property on $\Int$ is given again by \cite[Thm. 7.4.14]{cannarsa}, by exploiting the invariance of $\Int$, see \eqref{invInt}. Uniqueness of viscosity solutions on $\Int$ is an immediate consequence of the Comparison Principle given below (Theorem \ref{thmcomparison}). 

Point (2) is a classical verification theorem: if $V\in \mathcal{C}^{1,1}([0,T]\times S_d)$ then it is the unique classical solution of \eqref{HJB}, solving the equation also at boundary points. Under assumption \ref{B}, using in particular the strict convexity of $F$ in $a$, $a^*$ is the unique argmin of the pre-Hamiltonian \ref{preH} and thus the optimal control defined by \eqref{alphaopt} is unique. Note that dynamics \eqref{mu} is well-posed using the feedback $\alpha_*$ because $D^iV$ is Lipschitz-continuous. 
\end{proof}

It remains to show the comparison principle for viscosity solutions on $\Int$; to this end, it turns out that it is better to consider the dynamics in $\R^{d-1}$ and thus we state the equivalent result on $\mathrm{Int}(\hat{S}_d)$.   
In absence of boundary conditions in space, we must rely on the invariance of $\mathrm{Int}(\hat{S}_d)$ for the dynamics \eqref{dynx}. 
The following result then extends what we presented in \cite[Thm. 6.2]{CD} to a more general dynamics, and borrows ideas from the proofs  
of Theorem 3.8 and Proposition 7.3 in \cite{porrettaricciardi}.
We stress that we do not require here neither differentiability of the Hamiltonian nor convexity of the costs (in $a$ or $m$).

\begin{thm}[Comparison Principle on $\Intx$]
\label{thmcomparison}
Assume \ref{A} and let $u,v\in\mathcal{C}([0,T]\times \widehat{S}_d)$, $u$ be a  viscosity subsolution and $v$ be a  viscosity supersolution, respectively, of \eqref{HJBx} on $\mathrm{Int}(\hat{S}_d)$. 
If $u(T,x)\leq v(T,x)$ for any $x\in\widehat{S}_d$, then $u(t,x)\leq v(t,x)$ for any $t\in[0,T]$ and
$x\in\widehat{S}_d$.
\end{thm}


\begin{proof}
The idea is to define a supersolution $v_h$ that dominates $u$ at points near the boundary, for any $h$, and then use the comparison principle and pass to the limit in $h$. The parameter $h$ is needed to force $v_h$ to be infinity at the boundary of the simplex. Since the simplex has corners, the distance {to the boundary} is not a smooth function, so the first step is to construct a nice test function that goes to $0$ as $x$ approaches the boundary. Roughly speaking, we consider the product of the distances to the faces of the simplex, and then take its logarithm. 
\vskip 4pt

\emph{Step 1}. Let $\rho_i(x)$, for $x\in\Intx$,  be the distance of $x$ from the hyperplane $\{y \in {\mathbb R}^{{d-1}} : y_i=0\}$, for $i \in \ESd$, and $\rho_d(x)$ be the distance from $\{y \in {\mathbb R}^{{d-1}} : \sum_{l=1}^{{d-1}} y_l=1\}$. Specifically, for  $x\in\Intx$, we have
\[
\rho_i(x) = 
\begin{cases}
x_{i} \qquad & i \in \ESd,
\\
x^{-d}/
\sqrt{{d-1}}
\qquad& i=d,
\end{cases}
\]
where we recall that $x^{-d}=1 - \sum_{l \in \ESd} x_{l}$.
Clearly, $\rho_i\in \mathcal{C}^{\infty}(\hat{S}_d)$ and the derivatives, for $j\in\ESd$, are 
\[
\partial_{x_j}\rho_i(x) = 
\begin{cases}
\delta_{i,j} \qquad & i \in \ESd,
\\
 -1/\sqrt{{d-1}}
 \qquad& i=d.
\end{cases}
\]
Let us denote, for $x\in\mathrm{Int}(\hat{S}_d)$ and $z\in \R^{d-1}$, the Hamiltonian in \eqref{HJBx}
\be
\label{hatH}
\hat{\mathcal{H}}(t,x,w) = \sum_{i=i}^{d-1} x_j \hat{H}^j(t, x, w) + x^{-d} \hat{H}^d(t, x, w ),
\ee
where $\hat{H}^i$, for $i\in\ES$, are defined by \eqref{Hx}; note  that $\hat{\mathcal{H}}$ is convex in $w$.

\emph{Step 2}. For any $h>0$, let
\[
v_h(t,x) :=v(t,x) - h^2 \sum_{i \in \ES} \log \bigl(\rho_i(x)\bigr) + h (T-t), 
\qquad {(t,x) \in [0,T] \times \Intx}.
\]
We claim that $v_h$ 
is a viscosity supersolution of \eqref{HJBx} on $ \Intx$. 
Let then 
$\phi\in \mathcal{C}^1([0,T)\times \Intx)$, 
and  
$(\bar{t},\bar{x})\in [0,T)\times \Intx$ be  a local minimum of $v_h-\phi$ on 
$[0,T)\times \Intx$. Since $v$ is a viscosity supersolution of \eqref{HJBx} on $\Intx$, considering the test function 
$\phi_h \in\mathcal{C}^1([0,T)\times \Intx)$ given by $\phi_{h}(t,x)= \phi(t,x)+ h^2 \sum_{i \in \ES} \log (\rho_i(x)) - h (T-t)$, we get that
$(\bar{t},\bar{x})$ is a local minimum of $v-\varphi_h$ and thus
\[
-\partial_t \phi_h(\bar{t},\bar{x}) + \widehat{\mathcal{H}}\bigl(\bar{t},\bar{x}, D_x\phi_h(\bar{t},\bar{x})\bigr)
\geq 0,
\]
that is,
\be
\label{eq:78}
0\leq -\partial_t \phi(\bar{t},\bar{x}) -h  
+ \widehat{\mathcal{H}}\biggl(\bar{t},\bar{x}, D_x\phi(\bar{t},\bar{x})
+h^2 \sum\nolimits_{i \in \ES} \frac{D\rho_i(\bar{x})}{\rho_i(\bar{x})}
\biggr).
\ee
We denote $w= D_x\phi(\bar{t},\bar{x})$, $\check{w}^i = (w_1-w_i,\dots,w_{d-1}-w_i,-w_i)\in \R^d$ for $i\in\ESd$ and $\check{w}^d=(w,0)$,  $y_j = D\rho_j(\bar{x})$ for $j\in\ES$ and similarly $\check{y}^i_j$, and
$\varpi^i = \check{w}^i + h^2\sum_{j\in\ES} \frac{\check{y}^i_j}{\rho_j(\bar{x})}$, for $i\in\ES$.
 We apply the following property, which is an immediate consequence of the definition   of the Hamiltonian in \eqref{H}:  
\be 
\label{eq:79}
H^i(t,m, z+ \zeta)- H^i(t,m,z) \geq - \langle Q_{i,\bullet} (t, a^*(t,i,m,z), m) , \zeta \rangle
\ee 
for any $z,\zeta\in\R^d$ with $z_i=\zeta_i=0$, $i\in \ES$, $m\in \Int$, $t\in [0,T]$, where $a^*(t,i,m,z) \in A$ is an argmax of \eqref{preH}, which might not be unique. (This is equivalent to say that $-Q_{i,\bullet} (t, a^*(t,i,m,z), m)$ belongs to the subdifferential of the convex function $H^i(t,m,z)$.)

Choosing a maximizer $a^*(\bar{t}, i, \bar{x}, \varpi^i)$ for any $i\in\ES$, applying \eqref{eq:79} and \eqref{Hx}, we obtain (we omit the dependence on $(\bar{t},\bar{x})$)
\begin{align*}
&\widehat{\mathcal{H}}\biggl( w
+h^2 \sum\nolimits_{j \in \ES} \frac{y_j}{\rho_j(\bar{x})}
\biggr)\\
&= \sum_{i\in \ESd} \bar{x}_i H^i\bigg( \check{w}^i + h^2\sum_{j\in\ES} \frac{\check{y}^i_j}{\rho_j(\bar{x})} \Bigg) + \bar{x}^{-d}H^d\bigg( \check{w}^d + h^2 \sum_{j\in\ES} \frac{\check{y}^d_j}{\rho_j(\bar{x})} \Bigg)
\\
&\leq 
\sum_{i\in \ESd} \bar{x}_i H^i( \check{w}^i ) + \bar{x}^{-d}H^d( \check{w}^d ) \\
&\hspace{15pt} - h^2 \sum_{i\in \ESd} \bar{x}_i \sum_{j\in\ES} \Big\langle Q_{i,\bullet}(a^*(i,\varpi^i )) , \frac{\check{y}^i_j}{\rho_j(\bar{x})} \Big\rangle 
 - h^2  \bar{x}^{-d} \sum_{j\in\ES} \Big\langle Q_{d,\bullet}(a^*(d,\varpi^d ) ), \frac{\check{y}^d_j}{\rho_j(\bar{x})} \Big\rangle
\\
&= \hat{\mathcal{H}}(w) 
- h^2 \sum_{i\in \ESd} \bar{x}_i \sum_{j\in\ESd} \frac{1}{\rho_j(\bar{x})}\bigg[\sum_{k\in \ESd} Q_{i,k}(a^*(i,\varpi^i )) ( \delta_{k,j}-\delta_{i,j} ) + Q_{i,d}(a^*(i, \varpi^i))(-\delta_{j,i}) \bigg]
\\
& \hspace{15pt}- h^2 \sum_{i\in \ESd} \bar{x}_i  \frac{1}{\rho_d(\bar{x})}\bigg[\sum_{k\in \ESd} Q_{i,k}(a^*(i,\varpi^i )) ( -1/\sqrt{d-1}+ 1/\sqrt{d-1}) + Q_{i,d}(a^*(i, \varpi^i)) 1/\sqrt{d-1}
\bigg]
\\
& \hspace{15pt}
- h^2  \bar{x}^{-d} \sum_{j\in\ESd} 
\frac{1}{\rho_j(\bar{x})}\sum_{k\in \ESd} Q_{d,k}(a^*(d,\varpi^d )) \delta_{k,j}
- h^2  \bar{x}^{-d} \frac{1}{\rho_d(\bar{x})}\sum_{k\in \ESd} 
Q_{d,k} (a^*(d,\varpi^d )) (-1/\sqrt{d-1})
\\
&\leq \hat{\mathcal{H}}(w) 
+ h^2 \sum_{i\in \ESd} \bar{x}_i  \frac{1}{\rho_i(\bar{x})} M(d-1)
+h^2 \frac{\bar{x}^{-d}}{\rho_d(\bar{x}) \sqrt{d-1}} M(d-1)\\
&= \hat{\mathcal{H}}(w) 
+ h^2  M d (d-1),
\end{align*}
where we used the bound $0\leq Q_{i,j}\leq M$, for any $i\neq j \in \ES$, and the definition of $\rho_i$ in the last two lines. The latter inequality, applied in \eqref{eq:78}, gives
\[
-\partial_t \phi(\bar{t},\bar{x})  
+ \widehat{\mathcal{H}}\bigl(\bar{t},\bar{x}, D_x\phi(\bar{t},\bar{x})\bigr) 
\geq h- h^2Md(d-1) \geq 0 \qquad \mbox{ if } h\leq \frac{1}{Md(d-1)},
\]
which implies that $v_h$ is a viscosity supersolution of \eqref{HJBx} on $\Intx$ if $h$ is small enough.

\emph{Step 3}.
As $\rho_i\leq 1$, we have $v_h(t,x)\geq v(t,x)$ for any $(t,x) \in [0,T] \times \Intx$.  In particular, $v_h(T,x)\geq v(T,x)\geq u(T,x)$ for any $x \in \Intx$. 
We denote $\rho(x)=\prod_{i=1}^d \rho_i(x)$. Since $u$ and $v$ are bounded, we find that for any $h>0$ there exists $\eta>0$ (which may depend on $h$) such that 
$-h^2 \log \rho(x) \geq \|u\|_\infty + \|v\|_{\infty}$ if $\rho(x)\leq\eta$. We denote  
$\Gamma^\eta= \{x\in\widehat{\mathcal S}_d : \rho(x) =\eta\}$, 
$\mathcal{O}^\eta= \{x\in\widehat{\mathcal S}_d : \rho(x) \geq\eta\}$, and 
$\mathcal{O}^\eta_c= \{x\in\widehat{\mathcal S}_d : \rho(x) \leq\eta\}$; note that $\mathcal{O}^\eta$ is a smooth domain.   Thus 
$v_h(t,x)\geq u(t,x)$ for any $t\in[0,T]$ and $x\in \mathcal{O}^\eta_c$, in particular for any $x\in\Gamma^\eta$. Therefore we can apply the comparison principle 
(see  \cite[Thm. II.9.1]{fleso}) in $[0,T] \times \mathcal{O}^\eta$, because 
$u,v_h \in\mathcal{C}([0,T] \times \mathcal{O}^\eta)$: we obtain $u\leq v_h$ on $[0,T] \times \mathcal{O}^\eta$ and hence $u \leq v_{h}$  on the entire $[0,T] \times \widehat{\mathcal S}_{d}$, since we already have 
$u \leq v_{h}$ on $[0,T] \times 
\mathcal{O}^\eta_{c}$. 
%
Finally, we obtain $u\leq v$ on $[0,T]\times\Intx$ by sending $h$ to 0, as $\lim_{h\rightarrow 0}v_h(t,x) = v(t,x)$ for any $(t,x) \in [0,T] 
\times 
{\Intx}$, and then the inequality $u\leq v$ can be extended up to the boundary of $\hat{S}_d$ by continuity. 
\end{proof}

\subsection{Classical solution}
\label{sec3.2}

Here we give the sufficient condition for the value function of the MFCP to belong to $\mathcal{C}^{1,1}([0,T]\times S^d)$.
This is the the convexity assumption \ref{C}: we prove hence point (3) in Theorem \ref{thm:mainmfc}.

\begin{thm}
\label{thm18}
Under Assumption \ref{C},  
the value function is in $\mathcal{C}^{1,1}([0,T]\times S^d)$, and thus it is the unique classical solution to \eqref{HJB} (it solves the equation also at boundary points).
\end{thm}

We recall that a function $v:[0,T]\times S_d \rightarrow \R$ is called \emph{semiconcave}, resp. \emph{semiconvex}, on $[0,T]\times \Int$, with a constant $c$, if 
\be 
\frac{v(t+h,m+ p) -2v(t,x) + v(t-h,m-p)}{|h|^2 +  |p|^2} \leq c, \qquad \mbox{ resp. } \geq - c,
\ee 
for any $t\in [0,T]$, $m\in\Int$, $h$ with $t\pm h \in [0,T]$, and $p$ with $m\pm p \in \Int$. 

\begin{proof}
 We show that the value function is both semiconcave and semiconvex, in time and space, globally in $\mathrm{Int}(S^d)$, with a constant $c$; clearly it is equivalent to prove this propeties either for $V$ defined on $S_d$ or for $\hat{V}$ defined on $\hat{S}_d$. Recall again that $\Int$ is invariant for dynamics \eqref{mu}.  Hence Corollary  3.3.8 in \cite{cannarsa} ensures  that $V\in \mathcal{C}^{1,1}([0,T]\times \mathrm{Int}(S^d))$ and the Lipschitz constant of $D^i V$ is $c$, $\forall i\in\ES$. Thus in particular $V$ can be extended uniquely to a function in 
$\mathcal{C}^{1,1}([0,T]\times S^d)$, and then it solves \eqref{HJB} also at boundary points. The classical solution to \eqref{HJB} is unique because any solution is the value function, by a standard application of the verification theorem.

The value function is semiconcave on $[0,T]\times \Int$ by Theorem 7.4.11 in \cite{cannarsa}, thanks to Assumption (\ref{C}1); as above, to apply this result set on a Euclidean space, we have just to consider the equivalent formulation of the control problem on $\Intx$.  To prove that $V$ is semiconvex, we rewrite the MFCP in an equivalent formulation, with a control $w$,  such that the cost is  convex in $(m,w)$ and the dynamics is linear in $(m,w)$. Consider hence the problem of minimizing the cost 
\footnote{
This reformulation of the MFCP with the control $w$ is typically used for studying potential mean field games; see e.g. \cite{CGPT, BrianiCarda}.
}
\[
\widetilde{J}(w) := \int_0^T 
\sum_{i\in\ES} \mu^i_t f^i \left(t, \Big(\frac{w^{i,j}_t}{\mu^i_t} \Big)_{j\neq i} , \mu_t \right)
dt + \sum_{i\in \ES}g^i(\mu_T)\mu^i_{T},
\] 
where the couple $(\mu,w)$ satisfies the ODE
\[
\frac{d}{dt}\mu^i_t = \sum_j (w^{j,i}_t - w^{i,j}_t)
\]
and is subject to the constraints $w^{i,j}_t\geq 0$, $\frac{w^{i,j}_t}{\mu^i_t}\leq M$. If the initial condition $\mu_0\in \mathrm{Int}(S^d)$, then $\mu_t \in  \mathrm{Int}(S^d)$ for any $t\in[0,T]$, i.e. $\mu^i_t>0$, thus the cost is well-defined in this case. This control problem is indeed  well-defined only for $\mu_0\in \mathrm{Int}(S^d)$ and it is seen to be equivalent to the MFCP \eqref{mu}-\eqref{costmu} by setting 
$w^{i,j}=\mu^i \alpha^{i,j}$, meaning that the value function is the same.  The advantage in using this new formulation is that 
the dynamics is now linear in $w$ and the set of $(\mu,w)$ satisfying the constraints is convex. Moreover, the running cost is convex in $(m,w)$ and the terminal cost is convex in $m$ by Assumption (\ref{C}2).
Thus we can apply  \cite[Thm. 7.4.13]{cannarsa}, which says that  $V(t,m)$ is a convex function of $m\in \mathrm{Int}(S^d)$, for any $t\in [0,T]$.
 Then $V$ is semiconvex in time and space, in $[0,T]\times \mathrm{Int}(S^d)$, again by  \cite[Thm. 7.4.13]{cannarsa}, but using the original formulation of the MFCP in the proof therein, for which the coefficients are globally Lipschitz in $\mathrm{Int}(S^d)$, yielding thus global semiconvexity, while in the new formulation the cost is only locally Lipschitz in $m\in \mathrm{Int}(S^d)$.
\end{proof}

\subsection{Further properties and potential mean field games} 
\label{sec3.3}

We collect, for reference, other results concerning the mean field control problem and its relation, in some cases, with a mean field game. These are not used here, but might be useful for future works. They derive directly from the results of \cite[Section 7.4]{cannarsa} about deterministic control problems in a Euclidean space. Thus we consider the problem \eqref{dynx}-\eqref{costx} defined on $\hat{S}_d$ whose value function is denoted by $\hat{V}$. As before, for a function $G$ defined on $S_d$, we denote by $\hat{G}$ its version in local chart, i.e. $\hat{G}(x) = G(\check{x})$.


\begin{prop}
	\label{prop:solution:mfc:zero:epsilon}
	Assume \ref{A} and that $F(t, a, \cdot), Q_{i,j}(t,a,\cdot), G \in \mathcal{C}^1(S_d)$ and $\hat{\mathcal{H}}(t,\cdot,\cdot)\in\mathcal{C}^{1,1}(\hat{S}_d\times \R^{d-1})$, for 
	the Hamiltonian $\hat{\mathcal{H}}$  defined by \eqref{hatH}. 
	If $\alpha$ is an optimal control and $x$ the corresponding optimal trajectory, then there exists $w\in \mathcal{C}^1(0,T; \R^{d-1})$ such that
	\be
	\label{adjw}
\frac{d}{dt} w_t^j = \partial_{x_j} \hat{\mathcal{H}}(t,x_t, w_t), \qquad w_T^j= \partial_{x_j} \hat{G}(x_T);
\ee
	and $w_t$ belongs to the space superdifferential of $\hat{V}(t,x_t)$ for any time.

	Moreover, if $\hat{\mathcal H}(t,x, \cdot)$ is strictly convex for any $t\in[0,T]$ and $x\in \Intx$, and the costs $F$ and $G$ are semiconcave w.r.t. $m$, then, assuming that the control problem starts at $(t_0,x_0)\in [0,T]\times \Intx$,
	\begin{itemize}
	\item $\hat{V}$ is differentiable in $(t,x_t)$ for any $t\in(t_0,T]$, for any optimal trajectory $x$;
	\item
	$\hat{V}$ is differentiable in $(t_0,x_0)$ if and only if there exists a unique optimal trajectory $x$; in such case the adjoint process $w$ satisfies $w_t^j =\partial_{x_j} \hat{V}(t,x_t)$ for any $t\in [t_0,T]$. 
	\end{itemize}
\end{prop}

The Hamiltonian is strictly convex w.r.t. $w$, for $x$ in the interior, in case e.g.  the running cost is given by \eqref{quadratic}, as observed in \cite{CD}. We recall that the value function is (time-space) Lipschitz continuous, and thus differentiable almost everywhere for the 1-dim Lebesgue measure in time and the $(d\!-\!1)$-dim. Lebesgue measure in space. Further, $\hat{V}$ is shown to be (time-space) semiconcave in the proof of Theorem \ref{thm18}, in case only \ref{B} and (\ref{C}1) hold. 
The first assertion is instead the Pontryagin principle, which holds also under weaker assumptions. 

In case the cost splits as
\be 
\label{costsplit}
f^i(t,a,m) = \ell^i(t,a) +f^i_0(t,m),
\ee 
the Hamiltonian $H$ in \eqref{H} splits as   $H^i(t,m,z) = H^i_0(t,a) - f^i_0(t,m)$, and then 
\eqref{adjw} becomes
\be 
\label{eqw}
-\frac{d}{dt} w_t^j + \hat{H}^j_0(t,x_t,w_t) - \hat{H}^d_0(t,x_t, w_t) = \partial_{x_j} \hat{F}_0(t,x_t), \qquad w_T = \partial_{x_j} \hat{G}(x_T), \qquad j\in\ESd,
\ee 
where $F_0(m)= \sum_{i\in\ES} m_i f^i_0(m)$ and $\hat{H}$ is defines as in \eqref{Hx}.
The above equation is in fact equivalent to the HJB equation in the MFG system, first analyzed in \cite{gomes}, which is 
\be 
\label{equ}
-\frac{d}{dt} u_t^i + H^i_0(t,\mu_t,(u^j_t-u^i_t)_{j\in\ES})  = \mathfrak{f}^i(t,\mu_t), \qquad w^j_T = \mathfrak{g}^i_t(m_T), \qquad i\in\ES,
\ee 
for a given running cost $\mathfrak{f}$ and terminal cost $\mathfrak{g}$, in case \ref{B} holds, so that the transition rate in the coupled equation for  $\mu$ in \eqref{mu} are given by $\alpha^{i,j}_t = a^*_j (t,i, (u^j_t-u^i_t)_{j\in\ES})$.
Equivalence holds if the functions  $\mathfrak{f}$ and $\mathfrak{g}$ are such that
\be 
\label{pot}
\hat{\mathfrak f}^i(t,x)-\hat{\mathfrak f}^d(t,x) = \partial_{x_i} \hat{F}_0(t,x), \qquad \hat{\mathfrak g}^i(t,x)-\hat{\mathfrak g}^d(t,x) = \partial_{x_i} \hat{G}(t,x), \qquad i\in\ESd;
\ee  
for instance, the latter holds true by defining $\tilde{\mathfrak f}^i(t,x) = \partial_{x_i} \hat{F}_0(t,x)$ and $\tilde{\mathfrak f}^d(t,x)=0$. 
 It is  easy to see that \eqref{eqw} and \eqref{equ} are
equivalent: given a solution $u$ to \eqref{equ}, it suffices to let $w^j=u^j-u^d$, $j\in\ESd$; conversely, given $w$ solution to \eqref{eqw}, it suffices to solve \eqref{equ}
where all the occurrences of $u^i-u^j$  have been replaced by $w^i-w^j$  if 
$j\in\ESd$ and by $w^i$ if $j=d$. 
We remark that \eqref{adjw} can not be interpreted as a mean field game if the cost does not split as in \eqref{costsplit}. In case the cost splits, we have shown that, if an optimal control for the MFCP exists, then it gives rise to a solution of a particular mean field game --with costs determined by \eqref{pot}-- and therefore this can provide more informations on the optimal control and on the corresponding optimal trajectory of the MFCP. 

In general, a mean field game, in which hence the costs $(\mathfrak{f}^i)_{i\in\ES}$ and $(\mathfrak{g}^i)_{i\in\ES}$ are given, is said to be potential if \eqref{pot} holds and the cost $f_0^i$ and $g^i$ defining the MFCP do not depend on $i$ --say they are equal to $f_0$ and $g$-- so that $F_0=f_0$ and $G=g$. Thus the mean field game system represents the necessary conditions for optimality of the deterministic MFCP and we refer to \cite{CD} for a detailed study of potential mean field games and corresponding MFCP, in particular for the interpretation of \eqref{pot}.

\section{Convergence Results}
\label{sec4}

We prove here the main convergence results: Theorems \ref{thmconvV} and \ref{thmchaos}.
Throughout this section,   $V^N$ denotes the value function of the mean field $N$-agent optimization \eqref{costNmu}-\eqref{ratemu}, while $V$ denotes the value function of the mean field control problem 
\eqref{costmu}-\eqref{mu}. We recall that $V^N$ is the classical solution to ODE \eqref{HJBN}, while $V$ is the viscosity solution to PDE \eqref{HJB}.

\subsection{Convergence of value functions}
\label{Sec:4.1}

Here, we prove Theorem \ref{thmconvV}. Assume hence that Assumption \ref{A} is in force.
We exploit here the characterization of $V$ as the unique viscosity solution to \eqref{HJB} on the closed set $S_d$;
see Definition \ref{defvisco} and Proposition \ref{propvisco1}.

We first need to show that $V^N$ is time-space Lipschitz-continuous, uniformly in $N$; this is \eqref{eq:23} in Proposition \ref{prop5}. In this point, the compactness of the control set $A$ is required (Assumption (\ref{A}1)).
\begin{lem}
\label{lemlip}
 If \ref{A} holds, then for every $t,s \in [0,T]$ and $m,p \in S_d^N$ 
\be 
\label{lipVN}
|V^N (t, m) -V^N(s,p) | \leq C (|t-s|+ |m-p|),
\ee
for a constant $C$ independent of $N$.
\end{lem}  

\begin{proof}
We represent the dynamics of $\mu^N$, given by \eqref{ratemu}, as an SDE with respect to a Poisson random measure, as we did in \cite{cecfis}. 
We restrict attention to controls $\alpha:[0,T]\rightarrow A^d$ that are just functions of time and of the private state $i\in\ES$. Recall that, by Remark \ref{remeq}, the value function over this smaller class is the same $V^N$. 

Fix $N\in\mathbb{N}$ and denote by $\N(dt,d\theta)$ a standard Poisson random measure on $[0,T]\times [0,M]^{d\times d}$, with intensity measure $\nu$ on $[0,M]^{d\times d}$, where $M$ is the maximum of $Q$ (which is continuous over a compact set). Let $\nu$ be defined as the sum of the measures of the intersections with the axes, i.e. 
$\nu(E) = \sum_{i,j\in\ES} \mathrm{Leb} (E \cap \Theta_{i,j})$, where 
$\Theta_{i,j} := \{ \theta\in [0,M]^{d \times d} : \theta_{i',j'}=0     \mbox{ for all } 
(i',j')\neq (i,j) \}$, which is viewed as a subset of $\R$, as $E \cap \Theta_{i,j}$ is, and $\mathrm{Leb}$ is the Lebesgue measure on $\R$. Hence, the dynamics of $\mu^N$ 
is written as 
\be 
\label{NSDE}
 d \mu_t^N = \int_{[0,M]^{d\times d}} \frac1N (\delta_j-\delta_i) 
 \mathbbm{1}_{(0,N \mu^N_{i,t} Q_{i,j}(t, \alpha^i(t), \mu^N_t) ]} \N(dt, d\theta).
\ee 
This Equation is well-posed because $\mu^N$ takes values in $S^N_d$, which is finite, and then \eqref{ratemu} follows by \cite[Eq. (2.34)]{cecfis}. Observe that $\frac1N (\delta_j-\delta_i)$ represents the increment of $\mu^N$, while the indicator function gives the transition rate. 

Let $\mu^N$ start in $m$ and $\rho^N$ start in $p$, at a fixed time $t\in [0,T)$, with the same control $\alpha$. 
Let $\epsilon>0$ and $\alpha:[t,T]\rightarrow A^d$ be an $\epsilon$-optimal control for the problem starting at $(t,m)$, i.e. 
\[
J^N(t,m,\alpha)\leq \inf_{\beta} J^N(t,m,\beta) +\epsilon =V^N(t,m) +\epsilon,
\]
with an obvious notation for $J^N(t,m,\alpha)$. Then, by \cite[Lemma 3]{cecfis}, it follows that 
\begin{align*}
\E |\mu^N_s -\rho^N_s| &\leq |m-p| + \frac1N \sum_{i,j} |\delta_j-\delta_i| \E\int_t^s
| N \mu^N_{i,r} Q_{i,j}(r, \alpha^i(r), \mu^N_r) - N \rho^N_{i,t} Q_{i,j}(r, \alpha^i(r), \rho^N_r) | dr \\
& \leq |m-p| + C \int_t^s \E |\mu^N_r -\rho^N_r| dr,
\end{align*}  
by using the Lipschitz-continuity of $Q$, and thus Gronwall's lemma yields
\be 
\label{lipmuN}
\sup_{s\in [t,T]} \E |\mu^N_s -\rho^N_s| \leq C |m-p|. 
\ee 
Then 
\begin{align*}
V^N(t,p)-V^N(t,m) &\leq J^N(t,p,\alpha) - J^N(t,m,\alpha)  +\epsilon  \\
&\leq \E \int_t^T |F(s,\alpha(s), \mu^N_s) - F(s,\alpha(s), \rho^N_s)| ds 
+ \E |G(\mu^N_T) -G(\rho^N_T)|  +\epsilon  \\
&\leq C \sup_{s \in [t,T]} \E| \mu^N_s -\rho^N_s| +\epsilon \leq C|m-p| +\epsilon,
\end{align*} 
where we applied the Lipschitz-continuity of $F$ and $G$ (defined by \eqref{defFG}) and \eqref{lipmuN}. Taking the limit as $\epsilon$ vanishes, we get
$V^N(t,p)-V^N(t,m) \leq C|m-p|$ and then, changing the role of $m$ and $p$ we obtain also the opposite inequality, which provides
\be 
\label{lipVNmu}
|V^N(t,p)-V^N(t,m)| \leq C|m-p|.
\ee 

To prove the Lipschitz-continuity in time, note that \eqref{lipVNmu} implies 
$|D^{N,i}V^N(t,m)|\leq C$ for any $t\in[0,T]$ and $m\in S^N_d$, and thus, recalling that $V^N$ is $\mathcal{C}^1$ in time, from the HJB equation \eqref{HJBN} we derive that for any $t$ and $m$,
\begin{align*}
\left|\frac{d}{dt}V^N(t,m) \right| &= \left| \sum_{i\in \ES} m_i H^i(t, m, D^{N,i}V^N(t,m)) \right| \\
&\leq C(1+|D^{N,i}V^N(t,m)|) \leq C,
\end{align*} 
where we used the Lipschitz-continuity of $H$ w.r.t. $z$. Therefore
\[
\sup_{t\in[0,T], m\in S^N_d} \left|\frac{d}{dt}V^N(t,m) \right| \leq C,
\]
which, together with \eqref{lipVNmu}, yields \eqref{lipVN}.
\end{proof}

We now turn to the proof of Theorem \ref{thmconvV}.
As a first step to prove the convergence, it is required to extend the definition of $V^N$ outside $S_d^N$. One way to do this, similarly to \cite{ICD}, would be to consider the same control problem \eqref{costNmu}-\eqref{ratemu}, but starting at any point in the simplex, not only on $S_d^N$. However, in this way, the dynamics of $\mu^N$ would go also outside the simplex, and thus we prefer to not follow this strategy. Instead, we use a piecewise constant interpolation of $V^N$, and thus we have to pay a price in order to define it carefully at boundary points, so that the maximum in the usual doubling of variable argument (see \eqref{eq:63} below) is attained.

\begin{proof}[Proof of Theorem \ref{thmconvV}]

{ \ }

We first prove that 
	\be 
	\label{EN+}
	E_N^+ := \sup_{t\in [0,T], m\in S_d^N} 
	\left( V(t,m) - V^N(t,m) \right)\leq \frac{C}{\sqrt{N}}.
	\ee 
	If $E_N^+ \leq 0$ then \eqref{EN+} trivially holds, thus we assume that $E_N^+ >0$.
Since $V^N$ is defined on the grid $S^N_d$ only, we have to construct a piecewise constant extension $\tilde{V}^N$ defined on the whole $S_d$. Denote by $n(N,d)$ the number of elements in $S^N_d$, and cover the simplex by closed cells $(\Gamma_k)_{k=1}^{n(d,N)}$ centered in points of $S^N_d$ and invariant by translations. Note that the cells centered at points on the boundary cover also points outside the simplex. 
Define $\tilde{V}^N(p) = V^N(p_k)$ if $p\in\mathrm{Int}(\Gamma_k)$ and $p_k$ is the center of $\Gamma_k$. It remains to define $\tilde V^N$ at the boundaries of $\Gamma_k$. 

	
	\emph{Step 1}. 
	We exploit the usual argument of doubling the variables, which prompts us to consider
	the function, to be defined on $[0,T]^2 \times S_d ^2$,
	\be 
	\label{eq:63}
	\Phi(t,s,m,p):= V(t,m)-{\tilde V}^N(s,p)- \frac{|t-s|^2}{2\epsilon} 
	- \frac{|m-p|^2}{2\epsilon} - \frac{2T-t-s}{4T}E_N^+,
	\ee 
	where $\epsilon$ is a parameter to be fixed later in terms of $N$.
	Then the value of $\tilde{V}^N$ at the boundaries of the cells has to be chosen such that the above function admits a maximum. 
	Let us give first the idea of our construction.  The strategy is to define first $\tilde{V}^N$ constant in any closed cell, then perform the maximization in any closed cell --in which a maximum of $\Phi$ exists-- and thus take the maximum of the values obtained, so that  a maximum point for $\Phi$ exists.  If this maximum point lies in the interior of a cell, then there is no problem and  the value of $\tilde{V}^N$ at the boundaries does not matter. The critical situation is when the maximum point belongs to the boundary of a cell. In such situation, we have to define carefully $\tilde{V}^N$ at the boundary of the neighboring cells in order to verify equality \eqref{eq46} below. It is required because we want to exploit the ODE  \eqref{HJBN}, and this is indeed the main reason for considering a piecewise constant interpolation. We give below an example of our construction in case $d=2$, the generalization to $d>2$ being not difficult.
	
	Now, more precisely, for a cell $\Gamma_k$ centered at $p_k\in S^N_d$, $k\in\{1,\dots,n(N,d)\}$, let $\tilde{V}^N_k(p):=V^N(p_k)$ for any $p\in\Gamma_k$, thus also on the boundary of $\Gamma_k$.	
	Then define $\Phi_k$ as is \eqref{eq:63}, but with $p\in\Gamma_k$ and $\tilde{V}^N$ therein replaced by $\tilde{V}^N_k$, and let 
	\[
	\gamma_k:= \max_{[0,T]^2\times S_d \times (\Gamma_k\cap S_d)} \Phi_k(t,s,m,p).
	\]
	Such a maximum exists and is attained at a point $(\bar{t}_k,\bar{s}_k,\bar{m}_k,\bar{p}_k)$, which might  be non-unique and such that $\bar{p}_k$ belongs to the boundary of the cell $\Gamma_k$: we then let  
	\[
	\gamma= \max_{k=1,\dots,n(N,d)} \gamma_k =:\gamma_{\bar{k}}, \qquad (\bar{t},\bar{s},\bar{m},\bar{p}):= (\bar{t}_{\bar{k}},\bar{s}_{\bar{k}},\bar{m}_{\bar{k}},\bar{p}_{\bar{k}}). 
	\]
	We can now define  $\tilde V^N$ such that $\tilde V^N(\bar{s},\bar{p}_{\bar{k}}) = \tilde{V}^N(\bar{p}) := \tilde{V}^N_{\bar{k}} (\bar{p}) = V^N(\bar{s},p_{\bar{k}})$, the last equality holding by the definition above,   where we recall $p_{\bar{k}}\in S^N_d$ is the center of the cell $\Gamma_{\bar{k}}$.  Note that $\bar{p} = \bar{p}_{\bar{k}}$ may belong to the boundary of $\Gamma_{\bar{k}}$, in which case we have defined $\tilde{V}^N$ at the boundary of $\Gamma_{\bar{k}}$. 
	In addition, we require that 
	\be 
	\label{eq46}
	V^N\Big(\bar{s},p_{\bar{k}} + \frac1N(\delta_j-\delta_i)\Big)=\tilde{V}^N\Big(\bar{s},\bar{p} + \frac1N(\delta_j-\delta_i)\Big)
	\ee
	for any $i,j\in \ES$. Note that this is automatically satisfied if $\bar{p} \in \mathrm{Int}(S_d)$, while it gives the definition of $\tilde{V}^N$ on a part of the boundary of the cells bordering $\Gamma_{\bar{k}}$, if $\bar{p}$ belongs to the boundary of $\Gamma_{\bar{k}}$. 
	Lastly, $\tilde{V}^N$ can be defined
	arbitrarily right or left continuous at the other boundary points.
	
	Hence this construction guarantees that $\Phi$ in \eqref{eq:63} admits a maximum on $[0,T]^2\times S_d^2$ at $(\bar{t},\bar{s},\bar{m},\bar{p})$. We stress that $\tilde V^N$ is defined also outside the simplex, so that the RHS of \eqref{eq46} is meaningful in case $p_{\bar{k}} + \frac1N(\delta_j-\delta_i)$ belongs to the boundary of the simplex. 
	Note that $\tilde{V}^N$ is not continuous in space, but it remains Lipschitz in time, uniformly in space and w.r.t. $N$.
	

	
	Before proceeding with the rest of the proof, let us give an example of the above construction in case $d=2$. In this case the simplex is one dimensional and, if projected on $[0,1]$, we have
	$S^N_2= \{\frac{k}{N} : k=0,1,\dots,N\}$, the increments are $\pm \frac1N$ and the cells are $\Gamma_k= [ \frac{k}{N} -\frac{1}{2N} , \frac{k}{N} +\frac{1}{2N} ] $, for $k=0,\dots, N$. Let us assume that we are in the critical situation, that is, $\bar{p}$ belongs to the boundary of $\Gamma_{\bar{k}}$; for instance, we can assume that 
	$\bar{p}= \bar{p}_{\bar{k}} = \frac{\bar{k}}{N} +\frac{1}{2N}$.  
	The following picture provides then the construction of $\tilde{V}^N$ (where we omit the time dependence): 
	
	\hspace{2cm}
	\includegraphics[scale=0.5]{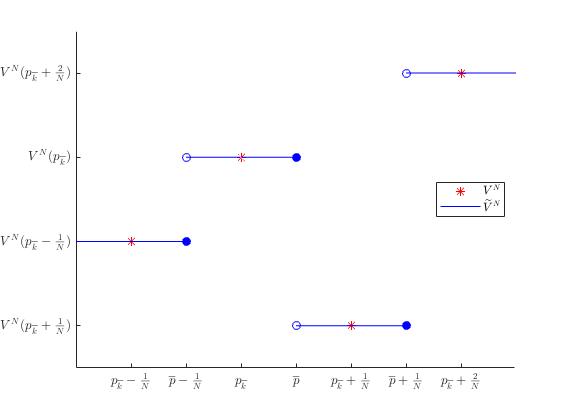}
	\newline	 \noindent
	Namely, if $\tilde{V}^N$ is left continuous at $\bar{p}$ then it is defined to be left continuous also at $\bar{p}\pm \frac1N$, so that 
	$V^N(p_{\bar{k}} \pm \frac1N) = \tilde{V}^N (\bar{p} \pm \frac1N)$. This latter property is crucial in the proof, see \eqref{eq49} below. 
	We remark again that $\tilde{V}^N$ can be define arbitrarily right or left continuous at the other boundary points. 
It is not difficult to generalize this construction to $d>2$.
	
	\emph{Step 2}. The inequality $\Phi(\bar{t},\bar{s},\bar{m},\bar{p})
	\geq \Phi(\bar{t},\bar{s},\bar{p},\bar{p})$ gives
	\[
	\frac{|\bar{m}-\bar{p}|^2}{2\epsilon} \leq V(\bar{t},\bar{m}) - V(\bar{t}, \bar{p}) ,
	\]
	which, using the Lipschitz continuity of $V$, yields
	\be 
	\label{28}
	|\bar{m}-\bar{p}|\leq C \epsilon. 
	\ee 
	Similarly, the inequality $\Phi(\bar{t},\bar{s},\bar{m},\bar{p})
	\geq \Phi(\bar{s},\bar{s},\bar{m},\bar{p})$ provides
	\[
	\frac{|\bar{t}-\bar{s}|^2}{2\epsilon} \leq V(\bar{t},\bar{m}) - V(\bar{s}, \bar{m}) + \frac{\bar{t}-\bar{s}}{4T} E_N^+,
	\]
	which, as $V$ is Lipschitz-continuous and $E_N^+$ is bounded, gives
	\be 
	\label{29}
	|\bar{t}-\bar{s}|\leq C \epsilon.
	\ee 
	We stress that \eqref{28} and \eqref{29} hold also in case 
	$(\bar{t},\bar{s},\bar{m},\bar{p})$ belong to the (time and space) boundary of $[0,T]^2\times S_d^2$.
	We fix now $\epsilon = \frac{1}{\sqrt{N}}$.
	
	\emph{Step 3}. 
	In order to prove \eqref{EN+}, we consider the three cases: either $\bar{t} =T$, or $\bar{s}=T$, or $\bar{t},\bar{s} <T$. 
	
	\textbf{First case:  $\bar{t} =T$ and $\bar{s} \in[0,T]$}.
	The inequlity $\Phi(\bar{t},\bar{s},\bar{m},\bar{p}) \geq
	\Phi(t,t,m,m)$, for any $(t,m)\in [0,T]\times S^N_d$, exploiting the Lipschitz-continuity in time of $\tilde{V}^N$ and in $S_d$ of $G$, gives
	\begin{align*}
	V(t,m) -\tilde{V}^N(t,m) &\leq G(\bar{m}) -G(\bar{p}) +G(\bar{p})- \tilde{V}^N(\bar{s},\bar{p}) 
	+ \frac{2T - 2t}{4T} E_N^+ 
	\\
	&\leq C|\bar{m} - \bar{p}| + C| T -\bar{s}| 
	+ \frac12 E_N^+.
	\end{align*}
	 Taking the supremum of the l.h.s. and applying \eqref{28} and \eqref{29}, we have
	 \[
	 E_N^+ \leq C\epsilon  + \frac12 E_N^+ \leq \frac{C}{\sqrt{N}} + \frac12 E_N^+ ,
	 \]
	which yields \eqref{EN+}. 
	
	\textbf{Second case: $\bar{s}=T$ and $\bar{t} \in[0,T]$}. This can be  treated as the first case, by using the Lipshitz-continuity in time of $V$. 
	
	\textbf{Third case:  $\bar{t} \in [0,T)$ and $\bar{s} \in [0,T)$}.
	Let $p_{\bar{k}}$ be, as above, the point in $S_d^N$ such that $\tilde{V}^N(\bar{s},p_{\bar{k}})=\tilde{V}^N(\bar{s},\bar{p})$. 
	Here we use the piecewise constant construction of $\tilde V^N$ and in particular \eqref{eq46}. 
	Hence, from \eqref{HJBN} we obtain 
	\be 
	\label{eq49}
	-\frac{d}{dt} \tilde{V}^N(\bar{s},\bar{p}) - 
	\sum_i p_{\bar{k}}^i \left\{
	\langle Q_{i,\bullet}(\bar{s},a^*_i,p_{\bar{k}})
	 , D^{N,i}\tilde{V}^{N}(\bar{s},\bar{p}) \rangle  + f^i(\bar{s} ,a^*_i,p_{\bar{k}})\right\}=0,
	\ee 
	where $a^*_i = a^*_i(\bar{s}, p_{\bar{k}})$ is a point in $A$ which attains the maximum in \eqref{preH}.
	We recall that $\tilde V^N$ is defined also in case $\bar{p} +\frac1N (\delta_j -\delta_i)$ is outside the simplex, which could happen if $p_{\bar{k}} +\frac1N (\delta_j -\delta_i)$ belongs to the boundary; while $p_{\bar{k}} +\frac1N (\delta_j -\delta_i)$ can be outside $S_d$ only if $p_{\bar{k}}^i=0$, in which case the term in \eqref{eq49} is zero. 
	
	Since $\tilde{V}^N-\varphi$ has a minimum at $(\bar{s},\bar{p})$, where
	\[
	\varphi(s,p)= V(\bar{t},\bar{m}) - \frac{|s-\bar{t}|^2}{2\epsilon} 
	- \frac{|p-\bar{m}|^2}{2\epsilon} - \frac{2T-s-\bar{t}}{4T}E_N^+,
	\]
	we get 
	$\frac{d}{dt} \tilde{V}^N(\bar{s},\bar{p})\geq 
	\partial_t\varphi(\bar{s},\bar{p})$, with equality if $\bar{s}\neq 0$, 	
	and thus 
	\be 
	\label{31}
	\frac{\bar{s}-\bar{t}}{\epsilon} - \frac{1}{4T}E_N^+
	- 
	\sum_i p_{\bar{k}}^i \left\{
	\langle Q_{i,\bullet}(\bar{s},a^*_i,p_{\bar{k}})
	 , D^{N,i}V^{N}(\bar{s},\bar{p}) \rangle + f^i(\bar{s} ,a^*_i,p_{\bar{k}})\right\}\geq 0.
	\ee
	On the other hand, as 
	$V-\psi$ has a maximum at $(\bar{t},\bar{m})$, where
	\[
	\psi(t,m) = \tilde{V}^N(\bar{s},\bar{p})+ \frac{|\bar{s}-t|^2}{2\epsilon} 
	+ \frac{|\bar{p}-m|^2}{2\epsilon} + \frac{2T-\bar{s}-t}{4T}E_N^+,
	\]
	since $V$ is a viscosity subsolution of \eqref{HJB} on the entire $S_d$ (see again Definition \ref{defvisco}) and $\psi$ is indeed defined on $\R^d$, we get
	\[ 
	-\frac{\bar{t}-\bar{s}}{\epsilon} + \frac{1}{4T}E_N^+ 
	+ \sum_i \bar{m}_i H^i  \left(\bar{t},\bar{m},\bigg(\frac{\langle \bar{m}-\bar{p}, \delta_j-\delta_i \rangle}{\epsilon}\bigg)_{j\in\ES} \right)
	\leq 0,
	\] 
	which, by definition of the Hamiltonian, gives 
	\be
	\begin{split}
	\label{33} 
	&-\frac{\bar{t}-\bar{s}}{\epsilon} + \frac{1}{4T}E_N^+ \\ 
	&- \sum_i \bar{m}_i 
	\bigg\{\bigg\langle Q_{i,\bullet}(\bar{t},a^*_i,\bar{m}) 
	 , \!\bigg(\frac{\langle \bar{m}-\bar{p} , \delta_j-\delta_i \rangle}{\epsilon}\bigg)_{j\in\ES} \bigg\rangle \!  + f^i(\bar{t} ,a^*_i,\bar{m})\bigg\}
	\leq 0.
	\end{split}
	\ee 
	The inequality
	$\Phi(\bar{t},\bar{s},\bar{m},\bar{p})
	\geq \Phi(\bar{t},\bar{s},\bar{m},\bar{p}+\frac1N (\delta_j-\delta_i))$ gives, for any $i$ and $j$ in $\ES$,  
	\begin{align*} 
	\tilde{V}^N\Big(\bar{s}, \bar{p} + \frac1N (\delta_j-\delta_i)\Big) 
	&\geq \tilde{V}^N(\bar{s},\bar{p}) + \frac{|\bar{m}-\bar{p}|^2}{2\epsilon}
	- \frac{|\bar{m}-\bar{p} - \frac1N (\delta_j-\delta_i)|^2}{2\epsilon}\\
	&\geq \tilde{V}^N(\bar{s},\bar{p}) + \frac{1}{N\epsilon} \bigg\langle \delta_j-\delta_i , \bar{m}-\bar{p} - \frac{1}{2N}(\delta_j-\delta_i) \bigg\rangle , 
	\end{align*} 
	which, applied in \eqref{31}, yields, as $Q_{i,j}\geq0$,
	\be 
	\label{34}
	\frac{\bar{s}-\bar{t}}{\epsilon} - \frac{1}{4T}E_N^+ \geq  
	\sum_i p_{\bar{k}}^i \bigg\{
	\bigg\langle Q_{i,\bullet}(\bar{s},a^*_i,p_{\bar{k}}) ,
	  \bigg(\frac{\langle \bar{m}-\bar{p} , \delta_j-\delta_i \rangle}{\epsilon}
	-\frac{1}{N\epsilon} 
	\bigg)_{j\in\ES} \bigg\rangle + f^i(\bar{s}, a^*_i, p_{\bar{k}} ) \bigg\}.
	\ee 
	Summing \eqref{33} and \eqref{34}, using Assumption \ref{A} (i.e. the boundedness of $Q$ and the Lipschitz-continuity of $Q$ and $F$  w.r.t. $(t,m)$), we obtain
	\begin{align*} 
	\frac{2}{4T}E_N^+ &\leq 
	\sum_i \bar{m}_i 
	\bigg\{ \bigg\langle Q_{i,\bullet}(\bar{s},a^*_i,\bar{m}) ,  \bigg(\frac{\langle \bar{m}-\bar{p} , \delta_j-\delta_i \rangle}{\epsilon}\bigg)_{j\in\ES} \bigg\rangle  + f^i(\bar{s} ,a^*_i,\bar{m})\bigg\}\\
	& \qquad -
	\sum_i p_{\bar{k}}^i \bigg\{ \bigg\langle  Q_{i,\bullet}(\bar{t},a^*_i,p_{\bar{k}}) ,  
	\bigg(\frac{\langle \bar{m}-\bar{p} , \delta_j-\delta_i\rangle}{\epsilon}
	-\frac{1}{N\epsilon} 
	\bigg)_{j\in\ES} \bigg\rangle
	+ f^i(\bar{t}, a^*_i,p_{\bar{k}}) \bigg\}\\
	&\leq C|\bar{t}-\bar{s}| + C |\bar{m}-p_{\bar{k}}| + C \frac{|\bar{m}-\bar{p}|^2}{2\epsilon} 
	+ \frac{C}{N\epsilon},
	\end{align*} 
	which provides, using \eqref{28}, \eqref{29} and that $|\bar{p}-p_{\bar{k}}|\leq\frac1N$ by construction, 
	\begin{align*} 
	E_N^+ &\leq C\left(|\bar{t}-\bar{s}| +   |\bar{m}-\bar{p}|  +\frac1N +\frac{|\bar{m}-\bar{p}|^2}{2\epsilon} 
	+ \frac{1}{N\epsilon} 
	\right) \leq C \left(3 \epsilon + \frac1N + \frac{1}{N \epsilon} \right) 
	\leq \frac{C}{\sqrt{N}},
	\end{align*} 
	recalling that  $\epsilon= \frac{1}{\sqrt{N}}$.
	
	\emph{Step 4}. The opposite inequality,
	\be 
	\label{EN-}
	E_N^- := \sup_{t\in [0,T], m\in S_d^N} 
	\left(  V^N(t,m) - V(t,m)\right)\leq \frac{C}{\sqrt{N}},
	\ee 
	can be proved in the same way, by changing the roles of $\tilde V^N$ and $V$, and using instead the viscosity supersolution property of $V$. Finally, \eqref{EN+} and \eqref{EN-} give \eqref{convV}.
\end{proof}

\begin{rem}
Similarly to \cite{ICD}, see also \cite[Sec. VI.1]{BCD}, it might be possible to establish a stronger convergence rate of order $1/N$, if $V^N$ is semiconcave uniformly in $N$:
\[
V^N(t, m+p)-2 V^N(t, m) + V^N(t,m-p) \leq c |p|^2,
\]
for a constant $c$ independent of $N$. Such estimate should hold if the costs $F$ and $G$ are semiconcave in space. This is left to future work. 
\end{rem}

As a consequence, we can also prove Theorem \ref{thmappr}.
In the proof as weel as in the next section, we use several times the following basic estimate: if $\bm{\xi}=(\xi_1,\dots,\xi_N)$ is a vector of $N$ i.i.d. random variables with values in $\ES$, such that $\mathrm{Law}(\xi_1)=m \in S_d$,  then 
\be 
\label{multi}
\E | \mu^N_{\bm{\xi}} - m| \leq \frac{1}{\sqrt{N}} \frac{\sqrt{d}}{2}.
\ee 
Indeed, $N\mu^N_{\bm{\xi}}= (N\mu^N_{\bm{\xi}}[i])_{i\in\ES}$ has a multinomial distribution with parameters $N$ and $m$, and thus the variance is $\E| N\mu^N_{\bm{\xi}}[i]-Nm_i|^2 = N m_i(1-m_i)$. Hence \eqref{multi} follows by Cauchy-Schwarz inequality and recalling that we are using the Euclidean norm.

\begin{proof}[Proof of Theorem \ref{thmappr}]
Let $\alpha^N(t,i,m)$ be an optimal feedback control for the $N$-agent optimization and $\alpha(t,i)$ be $\epsilon$-optimal for the MFCP, i.e. 
\be
\label{epsopt}
V(0,m_0)\leq J(\alpha) \leq V(0,m_0) +\epsilon.
\ee
Recall that the dynamics start at time 0 and the initial condition $\bm{X}_0$ is assumed to be i.i.d. with $\Law(X^1_0)= m_0$. 
Since $J^N(\alpha^N)= \E [V^N(0,\mu^N_0)]$, we have 
\begin{align*}
0&\leq J^N(\alpha) -J^N(\alpha^N) \\
&\leq |J^N(\alpha)- J(\alpha)| 
+ |J(\alpha)- V(0,m_0)| + |V(0,m_0) - \E[V(0,\mu^N_0)]| + |\E[V(0,\mu^N_0)] - \E[V^N(0,\mu^N_0)]| \\
& \leq |J^N(\alpha)-J(\alpha)| + \epsilon + C \E | \mu^N_0 -m_0| + \frac{C}{\sqrt{N}},
\end{align*}
where we used \eqref{convV}, \eqref{epsopt}, and the Lipschitz-continuity of $V$. As $\mu^N_0$ is the empirical measure of an i.i.d sequence, \eqref{multi} permits to bound $\E | \mu^N_0 -m_0|$.
Thus, in order to prove \eqref{appr}, it remains to show 
\be 
\label{eq56}
|J^N(\alpha)-J(\alpha)| \leq  \frac{C}{\sqrt{N}}.
\ee 
This follows by standard arguments in propagation of chaos. 
 
Let $\tilde{\rho}^N$ be the process given by dynamics \eqref{ratemu}, when using the decentralized control $\alpha(t,i)$. 
Assumption (\ref{A}3) gives (denoting $\alpha= (\alpha^1,\dots,\alpha^d)$ and using definition \eqref{defFG})
\begin{align*}
|J^N(\alpha)-J(\alpha)| &\leq \E \int_0^T \left|
F(t, \alpha(t), \tilde\rho^N_t) - F(t,\alpha(t), \mu_t)\right|dt + 
\E \left| G(\tilde\rho^N_t)-G(\mu_t) \right| \\
&\leq C\sup_{t\in [0,T]} \E|\tilde\rho^N_t - \mu_t|.
\end{align*}
Hence, to prove \eqref{eq56}, we have to show
\be 
\label{eq57}
\sup_{t\in [0,T]} \E|\tilde\rho^N_t - \mu_t| \leq \frac{C}{\sqrt{N}}.
\ee
We use the SDE  representation of the dynamics introduced in \eqref{NSDE}, thus $\tilde\rho^N$ solves the SDE
\[
d \tilde\rho_t^N = \int_{[0,M]^{d\times d}} \frac1N (\delta_j-\delta_i) 
 \mathbbm{1}_{(0,N \tilde\rho^N_{i,t} Q_{i,j}(t, \alpha^i(t), \tilde\rho^N_t) ]} \N(dt, d\theta).
\]
This process is not the empirical measure of a vector of independent processes because of the dependence of $Q$ on $m$, thus we introduce also the process $\tilde\mu^N$ in which the empirical measure is replaced by the limit deterministic flow $\mu$: $\tilde\mu^N$ solves
\[
d \tilde\mu_t^N = \int_{[0,M]^{d\times d}} \frac1N (\delta_j-\delta_i) 
 \mathbbm{1}_{(0,N \mu_{i,t} Q_{i,j}(t, \alpha^i(t), \mu_t) ]} \N(dt, d\theta).
\]
and represents the empirical measure associated to $N$ i.i.d copies of the limit process $\tilde{X}$ satisfying \eqref{limX} and such that $\Law (\tilde{X}_t)= \mu_t$. Again \eqref{multi} gives
\be 
\label{eq58}
\sup_{t\in [0,T]} \E| \tilde\mu^N_t - \mu_t| \leq  \frac{C}{\sqrt{N}}.
\ee 
Thus it remains to estimate the distance between $\tilde\rho^N$ and $\tilde \mu^N$. As in the proof of Lemma \ref{lemlip}, the representation in terms of SDEs yields, by exploiting the 
Lipschitz-continuity of $Q$ w.r.t. the variable $m$,
\begin{align*}
\E |\tilde\rho^N_t -\tilde\mu^N_t| &\leq  \frac1N \sum_{i,j} |\delta_j-\delta_i| \E\int_0^t
| N \tilde\rho^N_{i,s} Q_{i,j}(s, \alpha^i(s), \tilde\rho^N_s) - N \mu_{i,t} Q_{i,j}(s, \alpha^i(s), \mu_s) | ds \\
& \leq C \int_0^t \E |\tilde \rho^N_s -\mu_s| ds
\leq C \int_0^t \E |\tilde \rho^N_s -\tilde\mu^N_s| ds + C \sup_{t\in [0,T]} \E| \tilde\mu^N_t - \mu_t| \\
&\leq C\int_0^t \E |\tilde \rho^N_s -\tilde\mu^N_s| ds + \frac{C}{\sqrt{N}}.
\end{align*} 
Therefore Gronwall's inequality gives $\sup_{t\in [0,T]} \E| \tilde\rho^N_t - \tilde\mu^N_t| \leq  \frac{C}{\sqrt{N}}$, which, together with \eqref{eq58}, provides \eqref{eq57} that  concludes the proof. 
\end{proof}

\subsection{Propagation of chaos}

Here we prove Theorem \ref{thmchaos}. 
Throughout this subsection, we hence assume that $V\in \mathcal{C}^{1,1}([0,T]\times S_d)$ and that Assumptions \ref{B} and (\ref{C}1) are in force.


We first show the following simple result. 

\begin{lem}
	There exists a constant $C$ such that for any $i\in\ES$
	 \be 
	 \label{20}
	\max_{t\in [0,T], m\in S_d^N} 
	\left|D^{N,i}V(t,m) -D^i V(t,m)\right| \leq \frac CN.
	 \ee 
\end{lem}

\begin{proof}
	By definition, for each $i,j\in\ES$ and any $m\in S_d^N$ (we omit the time in the notation), using Taylor's formula and the 
	Lipschitz-continuity of $D^i V$, we obtain
	\begin{align*}
		&\left|D^{N,i}V(m) -D^i V(m)\right| =
	\left| N \bigg[ V \bigg( m+\frac1N(\delta_j-\delta_i)\bigg) -V(m)\bigg]
	- \partial_{m_j - m_i}V(m) \right|  \\
	&= \bigg| \int_0^1 \partial_{m_j - m_i} V \bigg( m + \frac1N s (\delta_j-\delta_i)\bigg)ds - \partial_{m_j - m_i} V(m) \bigg|\\
	& \leq \frac CN \int_0^1 s |\delta_j-\delta_i| ds
	\leq \frac CN.
	\end{align*}
\end{proof}



As in the statement of Theorem \ref{thmchaos}, let 
 $\alpha^N$ be the unique optimal feedback control for the $N$-agent optimization defined by \eqref{optalN}, and  $\mu^N$ be the corresponding optimal process satisfying \eqref{muNopt}. 
 Also, let $\alpha_*$ be the unique optimal feedback control for the MFCP defined by \eqref{alphaopt} and $\mu$ the corresponding optimal trajectory given by \eqref{mu}. We stress that $\alpha_N$ and $\alpha_*$ are functions of $t$ and $m$. Since Assumption (\ref{C}1) is in force, as explained in Remark \ref{remk}, we can assume here that the convergence of the value functions holds with a stronger rate, given by \eqref{convK}.

As an intermediate step, we consider  the process $\rho^N$ satisfying \eqref{muNopt} with the limiting feedback control $\alpha_*$; such intermediate process is needed to prove convergence and describes the empirical measure of a standard mean field interacting particle system, in the sense that the transition rate function $\alpha_*$ is the same for any $N$ (and depends on the empirical distribution $\rho^N$), while in $\mu^N$ the rate $\alpha^N$ depends on $N$. 
Thus, we first show the proximity of $\mu^N$ and $\rho^N$ and then prove convergence of $\rho^N$ to $\mu$. 
We assume that $\rho^N$ and $\mu^N$ start at time zero in the same point in $S^N_d$. 
The proof of the following proposition is the point where assumption
(\ref{B}3) is required.

\begin{prop}
\label{propNrho}
We have 
\be 
\label{murhoN}
\E \bigg[\sup_{0\leq t\leq T}|\mu^N_s -\rho^N_s|\bigg] \leq \frac{C}{\sqrt{N}}.
\ee 
\end{prop}

\begin{proof}
	 \emph{Step 1}. We compute the limit value function $V$ along $\mu^N$. Dynkin formula and then the HJB equation \eqref{HJB} give 
	\begin{align*}
	&\E[V(T,\mu^N_T)] -\E[V(0,\mu^N_0)] \\
	&=
	\E\Bigg[\int_0^T \bigg(\partial_t V(t, \mu^N_t) 
	+ N \sum_{i\in \ES} 
	\mu^N_{i,t} \alpha_N^{i,j}(t,\mu^N_t) 
	\Big[V \Big(t,\mu^N_t + \frac1N (\delta_j-\delta_i)\Big) - V(t,\mu^N_t) \Big] \bigg)dt \Bigg] \\
	&= \E\Bigg[\int_0^T  \sum_{i\in\ES} \mu^N_{i,t} \Big(
	H^i(t, \mu^N_t, D^i  V(t,\mu^N_t)) 
	+ \langle \alpha^i_N(t,\mu^N_t) , D^{N,i}V(t,\mu^N_t) \rangle \Big)dt \Bigg]\\
	&= \E\Bigg[\int_0^T  \sum_{i\in\ES} \mu^N_{i,t} \Big(
	 \mathcal{H}^i \big(t, a^*(t,i,\mu^N_t, D^i V(t,\mu^N_t)), \mu^N_t, D^i V(t,\mu^N_t)\big) - 
	 \mathcal{H}^i\big(t, \alpha_N^i(t,\mu^N_t), \mu^N_t,D^i V(t,\mu^N_t)) \\
	 &\qquad 
	-  f^i(t, \alpha_N^i (t,\mu^N_t), \mu^N_t)
	+ \langle \alpha^i_N(t,\mu^N_t) , D^{N,i}V(t,\mu^N_t) 
	- D^i V(t,\mu^N_t) \rangle \Big) dt \Bigg].
	\end{align*} 
	Since $a^*(t,i,m,z)$ maximizes the pre-Hamiltonian $\mathcal{H}^i(t,m,z)$ in \eqref{preH}, using \eqref{conva}, we obtain, for any $a\in [0,M]^d$, 
	\be 
	\mathcal{H}^i(t, a^*(t,i,m,z)  ,m,z) -\mathcal{H}^i(t,a,m,z) \geq \lambda |a^*(t,i,m,z) -a|^2,
	\ee 
	This inequality, together with \eqref{20} and the uniform boundedness of $\alpha^N$, yields 
	\begin{equation*}
	\begin{split}
	\E\Bigg[  &\int_0^T \sum_{i\in \ES} \mu^{N}_{i,t} f^i(t, \alpha_N^i (t,\mu^N_t), \mu^N_t)dt + \sum_{i\in \ES}g^i(\mu^N_T)\mu^{N}_{i,T}  \Bigg] - \E[V(0,\mu^N_0)] \\
	&
	\geq \E\Bigg[  \int_0^T \lambda \sum_{i\in\ES} \mu^N_{i,t} \left| \alpha^i_N(t,\mu^N_t) - a^*(t,i,\mu^N_t,D^iV(t,\mu^N_t))\right|^2 dt \Bigg]
	- \frac CN.
	\end{split}
	\end{equation*}
	Observe that  on the l.h.s. there is exactly $J^N(\alpha_N)$, as defined by \eqref{costNmu}, which is equal to $\E[V^N(0,\mu^N_0)]$. Therefore we obtain
	\[
	\E\Bigg[\!  \int_0^T \!\! \sum_{i\in\ES} \! \mu^N_{i,t} \left| \alpha^i_N(t,\mu^N_t) - a^*(t,i,\mu^N_t,D^iV(t,\mu^N_t))\right|^2 \! dt\Bigg]
	\!\leq\! \frac CN + 
		C\! \max_{ m\in S_d^N} 
	\left|V^N(0,m) - V(0,m)\right|\! ,
		\]
		which, applying \eqref{convK}, gives
		\be 
		\label{boundcontrols}
		\E\Bigg[  \int_0^T  \sum_{i\in\ES} \mu^N_{i,t} \left| \alpha^i_N(t,\mu^N_t) - a^*(t,i,\mu^N_t,D^iV(t,\mu^N_t))\right|^2 dt \Bigg]
		\leq \frac{C}{N}.
		\ee 
		
	\emph{Step 2}. Considering now the process $\rho^N$, and applying the SDE representation as in the proof of Lemma \ref{lemlip}, we obtain  
	 for any $t>0$ 
	\begin{align*}
	\phi(t)&:=\E \bigg[\sup_{0\leq s\leq t}|\mu^N_s -\rho^N_s|\bigg] \leq \E \Bigg[\int_0^t \sum_{i,j} \frac1N(\delta_j-\delta_i) \left|N\mu^N_{i,s} \alpha_N^{i,j}(s,\mu^N_s) 
	- N\rho^{N}_{i,s} \alpha^{i,j}_*(s,\rho^N_s) \right| ds\Bigg]  \\
	&\leq 2	\E \Bigg[\int_0^t \sum_{i,j} \bigg( \mu^N_{i,s} \left|\alpha_N^{i,j}(s,\mu^N_s) - \alpha^{i,j}_*(s,\mu^N_s)\right|
	+ \mu^N_{i,s} \left|\alpha^{i,j}_*(s,\mu^N_s) - \alpha^{i,j}_*(s,\rho^N_s)\right| \\
	&\qquad
	 + \alpha^{i,j}_*(s,\rho^N_s) \left|\mu^N_{i,s}
	- \rho^{N}_{i,s}  \right| \bigg) ds\Bigg]
	\end{align*}
	and applying  Jensen's inequality, the boundedness and Lipschitz-continuity of the feedback function $\alpha_*$ (verified as $V$ is $\mathcal{C}^{1,1}$ and $\alpha_*$ is given by \eqref{alphaopt}),  and  then \eqref{boundcontrols},  we get
	\begin{align*}
	\phi(t) &\leq C \sqrt{ \E\bigg[  \int_0^t  \sum_{i\in\ES} \mu^N_{i,s} \left| \alpha^N_i(s,\mu^N_s) - a^*(t,i,\mu^N_s,D^iV(t,\mu^N_s))\right|^2 dt \bigg] }
	+C \E \left[\int_0^t \left| \mu^N_{s}
	- \rho^{N}_{s} \right| ds \right] \\
	& \leq \frac{C}{\sqrt{N}} + C \int_0^t \phi(s)ds,
	\end{align*}
	which, by Gronwall's lemma, provides \eqref{murhoN}.
\end{proof}

We introduce also the process $\bm{Y}$ related to the empirical measure $\rho^N$, in which the transition rate function is given in terms of the limiting optimal feedback $\alpha_*(t,i,m)$, independent of $N$. In order to prove the propagation of chaos result (Theorem \ref{thmchaos}) we first show the proximity of $\bm{X}$ and $\bm{Y}$ and then prove propagation of chaos for the process $\bm{Y}$, which is a standard mean field interacting particle system. The proof is pretty standard and the arguments are the same that we used in \cite{CP}, based on a probabilistic representation of the dynamics, thus we are not going to give all the details below. 
We recall that the initial conditions are fixed and i.i.d.. The following is the counterpart of Proposition \ref{propNrho} for the $N$ trajectories.

\begin{lem}
For any $N\in\mathbb{N}$ and $k\in\NN$, it holds
\be 
\label{distXY}
\E \bigg[\sup_{0\leq t\leq T}|X^k_s -Y^k_s|\bigg] \leq \frac{C}{\sqrt{N}}.
\ee 
\end{lem} 

\begin{proof}
By exploiting the representation of $\bm{X}$ and $\bm{Y}$ in terms of SDEs driven by Poisson random measures (similar for the representation for $\mu^N$ used above, see \cite{CP} for the details), we obtain
\begin{align*}
\varphi(t)&:=\mathbb{E} \left[\!\sup_{s \in [0,t]} \! |X^k_{s} - Y^k_{s}|\right] 
\leq \! C \mathbb{E} \! \int_{0}^t \left(\left|\alpha^N(s,X^k_s, \mu^N_s) - \alpha_*(s, Y^k_s, \rho^N_s)\right| + |X^k_{s} - Y^k_{s}|\right) ds
\\
&\leq C \mathbb{E} \! \int_{0}^t \left(\left|\alpha^N(s,X^k_s, \mu^N_s) - \alpha_*(s, X^k_s, \rho^N_s)\right| + |X^k_{s} - Y^k_{s}|\right) ds,
\end{align*} 
where we used that any bounded function is Lipschitz with respect to $i\in\ES$ finite.
Using the exchangeability of the processes $(\bm{X},\bm{Y})$, we can rewrite
\begin{align*}
\mathbb{E} & \int_{0}^t \left|\alpha^N(s,X^k_s, \mu^N_s) - \alpha_*(s, X^k_s, \rho^N_s)\right| ds = 
\frac1N \sum_{l=0}^N \mathbb{E} \! \int_{0}^t \left|\alpha^N(s,X^l_s, \mu^N_s) - \alpha_*(s, X^l_s, \rho^N_s)\right| ds \\
& = \frac1N \sum_{l=0}^N \mathbb{E} \! \int_{0}^t \sum_{i=1}^d \mathbbm{1}_{\{X^l_s=i\}} \left|
\alpha^N(s,i, \mu^N_s) -  \alpha_*(s, i, \rho^N_s)\right| ds \\ 
& = \mathbb{E} \! \int_{0}^t \sum_{i=1}^d \mu^N_{i,s}
\left| \alpha^N(s,i, \mu^N_s) - \alpha_*(s, i, \rho^N_s)\right| ds
\leq \frac{C}{\sqrt{N}},
\end{align*}
the latter bound following from the proof of the previous proposition.
Therefore we derive
\[
 \phi(t) \leq  \frac{C}{\sqrt{N}} + \int_0^t \phi(s) ds,
\]
which gives \eqref{distXY} by Gronwall's lemma.

\end{proof}



We are finally in the position to prove Theorem \ref{thmchaos}.
Recall that $\tilde{\bm{X}}$ is the decoupled process, made by independent copies of the limit dynamics, in which player $\tilde{X}^k$ uses the decentralized strategy $\alpha_*(t, \tilde X^k_t, \mu(t)) $, where $\alpha_*$ is the optimal control and $\mu$ is the optimal trajectory for the MFCP, and we have 
$\mathrm{Law} (\tilde{X}^k_t) = \mu(t)$ for any $t\in[0,T]$ and $k\in\NN$.


\begin{proof}[Proof of Theorem~\ref{thmchaos}] 

\noindent
Thanks to \eqref{distXY}, claim \eqref{chaosXtilde}  is proved if we show that
\begin{equation}
\label{last}
\mathbb{E}\left[\sup_{t \in [0,T]} |Y^k_{t} - \widetilde{X}^k_{t}|\right] \leq \frac{C}{\sqrt{N}}.
\end{equation}
Let $\tilde{\mu}^N_t$ be the empirical measure related to the i.i.d. process $\tilde{\bm{X}}$. 
Again \eqref{multi} gives
\be 
\sup_{t\in[0,T]}\mathbb{E} \left|\tilde\mu^N_t - \mu_t\right| \leq \frac{C}{\sqrt{N}}.
\ee 
The limit feedback $\alpha_*$ is Lipschits w.r.t. $m$ because so is $D V$. We use this fact and also the inequality
\be
\label{boundwass} 
|\mu^N_{\bm{x}} - \mu^N_{\bm{y}}| \leq C \frac1N \sum_{k=1}^N |x_k - y_k|
\ee
which follows from the definition of the 1-Wasserstein distance, which is is equivalent to the Euclidean distance in finite dimension. 
As in the proof of the previous lemma, we have
\begin{align*}
\phi&(t) := \mathbb{E}\left[\sup_{s \in [0,t]} |Y^k_{s} - \widetilde{X}^k_{s}|\right] \\
& \leq \mathbb{E}\left[\int_{0}^t \left|
\alpha_*(s, Y^k_s, \rho^N_s) - \alpha_*(s, Y^k_s, \mu_s)
\right| ds + C\int_{0}^t \left|Y^k_{s} - \widetilde{X}^k_{s}\right|ds\right]\\
& \leq C\mathbb{E} \int_{0}^t \left|
 \rho^N_s - \mu_s
\right| ds + C\E \int_{0}^t \left|Y^k_{s} - \widetilde{X}^k_{s}\right|ds \\
&\leq C\mathbb{E} \int_{0}^t \left|
 \rho^N_s - \tilde\mu^N_s
\right| ds  
+ C \sup_{t\in[0,T]}\mathbb{E} \left|\tilde\mu^N_t - \mu_t\right| 
+ C\E \int_{0}^t \left|Y^k_{s} - \widetilde{X}^k_{s}\right|ds \\
&\leq \frac1N \sum_{l=1}^N C\E \int_{0}^t \left|Y^l_{s} - \widetilde{X}^l_{s}\right|ds + \frac{C}{\sqrt{N}} + C\int_0^t \phi(s) ds \\
&\leq \frac{C}{\sqrt{N}} + C\int_0^t \phi(s) ds,
\end{align*}
where in the last inequality we used the exchangeability of the processes. 
Thus Gronwall's inequality yields \eqref{last}. 

Finally, observe that, if (\ref{C}1) is not in force --and thus just \eqref{convV} is satisfied and not \eqref{convK}-- then \eqref{boundcontrols} holds with $N$ replaced by $\sqrt{N}$ and hence, as a consequence of the above proofs, estimates \eqref{murhoN} and \eqref{last} hold with $\sqrt{N}$ replaced by $N^{1/4}$. Thus \eqref{boundwass} gives 
\[
\E\left[ \sup_{t\in [0,T]} | \rho^N_t - \tilde{\mu}^N_t | \right] \leq \frac{C}{N^{1/4}}.
\]  
Therefore  
\eqref{chaosmu} follows from \eqref{murhoN} (with $\sqrt{N}$ replaced by $N^{1/4}$ therein), the above estimate, and the inequality
\[
\mathbb{E}\left[\sup_{t \in [0,T]} \left|\tilde\mu_t - m_t\right|\right]  \leq C N^{-1/9},
\]
which can be found in \cite{RR}.
\end{proof}


\begin{thebibliography}{9}

\bibitem{BCD} M. Bardi and I. Capuzzo-Dolcetta. \emph{Optimal Control and Viscosity Solutions of Hamilton-Jacobi-Bellman Equations}. Birkh\"auser, Boston, 1997.

\bibitem{BCCD-N} E. Bayraktar, A. Cecchin, A. Cohen, and F. Delarue. Finite state mean field games with Wright-Fisher common noise as limits of $N$-player weighted games. 
Preprint 2020, arXiv:2012.04845.

\bibitem{BC} E. Bayraktar and A. Cohen. Analysis of a finite state many player game using its master equation. \emph{SIAM Journal
on Control and Optimization}, 56(5):3538-3568, 2018. 

\bibitem{BCP} E. Bayraktar, A. Cosso and H. Pham. Randomized dynamic programming principle and Feynman-Kac representation for optimal control of McKean-Vlasov dynamics. \emph{Transactions of the American Mathematical Society}, 370: 2115-2160, 2018. 

\bibitem{Bensoussan}
A. Bensoussan, J. Frehse, and S. Yam. \emph{Mean field games and mean field type control theory} , volume 101 of Springer Briefs in mathematics. Springer–Verlag New York, 2013.

\bibitem{BrianiCarda} A. Briani and P. Cardaliaguet. Stable solutions in potential mean field game systems.  \emph{Nonlinear Differential Equations and Applications}, 25:1, 2018.


	\bibitem{cannarsa} P. Cannarsa and C. Sinestrari. \emph{Semiconcave Functions, Hamilton-Jacobi Equations, and Optimal Control}. Birkhauser, Boston, 2004. 
	
	\bibitem{ICD}  I. Capuzzo-Dolcetta and H. Ishii. Approximate solutions of the Bellman
equation of deterministic control theory. \emph{Applied Mathematics and Optimization}, 11:161-181,
1984.

\bibitem{CDLL} P. Cardaliaguet, F. Delarue, J.-M. Lasry, and P.-L. Lions. \emph{The master equation and the convergence problem
in mean field games}, volume 201 of \emph{Annals of Mathematics Studies}. Princeton University Press, Princeton, NJ,
2019.
	
	\bibitem{CGPT} P. Cardaliaguet, J. Graber, A. Porretta, and D. Tonon. Second order mean field games with degenerate
diffusion and local coupling. \emph{Nonlinear Differential Equations and Applications}, 22:1287-1317, 2015.
	
	\bibitem{CarDelMKV} R. Carmona and F. Delarue. Forward-backward stochastic differential equations and
controlled Mckean Vlasov dynamics. \emph{The Annals of Probability}, 43:2647–2700, 2015.
	
	\bibitem{CarDel1} R. Carmona and F. Delarue. \emph{Probabilistic Theory of Mean Field Games with Applications:
Vol. I, Mean Field FBSDEs, Control, and Games}. Volume 83 of Probability Theory and
Stochastic Modelling, Springer, 2018.

\bibitem{CarDel2} R. Carmona and F. Delarue. \emph{Probabilistic Theory of Mean Field Games with Applications:
Vol. II, Mean Field Games with Common Noise and Master Equations}. Volume 84 of
Probability Theory and Stochastic Modelling, Springer, 2018.
	
	
\bibitem{CDFP} A. Cecchin, P. Dai Pra, M. Fischer, and G. Pelino. On the convergence problem in mean field games: a
two state model without uniqueness. \emph{SIAM Journal on  Control and  Optimization}, 57(4):2443-2466, 2019.	
	
	\bibitem{CD} A. Cecchin and F. Delarue. Selection by vanishing common noise for potential finite state mean field games. Preprint 2020, arXiv:2005.12153.
	
	\bibitem{cecfis} A. Cecchin and M. Fischer. Probabilistic approach to finite state mean field games. \emph{Applied Mathematics and Optimization}, 81: 253-300, 2020.
	
	\bibitem{CP} A. Cecchin and G. Pelino. Convergence, fluctuations and large deviations for finite state mean field games via the master equation. \emph{Stochastic Processes and their Applications}, 129(11): 4510-4555, 2019.
	
	\bibitem{CL} M. G. Crandall and P. L. Lions. Two approximations of solutions of Hamilton-Jacobi equations. \emph{Mathematics of Computations}, 43(167): 1-19, 1984.
	
	\bibitem{DelFT} F. Delarue and R. Foguen Tchuendom. Selection of equilibria in a linear quadratic mean-field game.
\emph{Stochastic Processes and their Applications}, 130(2): 1000-1040, 2020. 
	
\bibitem{djete} M. F. Djete. Extended mean field control problem: a propagation of chaos result. Preprint 2020,  	arXiv:2006.12996.	
	
	\bibitem{DPT1} F. M. Djete, D. Possamai, and X. Tan. McKean-Vlasov optimal control: the dynamic programming principle. Preprint 2020,  	arXiv:1907.08860. 
	
\bibitem{DPT} F. M. Djete, D. Possamai, and X. Tan.
McKean-Vlasov optimal control: limit theory and equivalence between different formulations. Preprint 2020, arXiv:2001.00925.	
	
	\bibitem{Fischer} M. Fischer. On the connection between symmetric $N$-player games and mean field games. \emph{Annals of Applied Probability}, 127(2):757–810, 2017.
	
	\bibitem{fleso} W. H. Fleming and H. M. Soner. \emph{Controlled Markov Processes and Viscosity Solutions},
		volume 25 of \emph{Applications of Mathematics}. Springer, New York, 2nd edition, 2006.
		
	
	\bibitem{FLOS} M. Fornasier, S. Lisini, C. Orrieri, and G. Savar\'e. Mean-field optimal control as Gamma-limit of finite agent controls. \emph{European Journal of Applied Mathematics}, 30(6): 1153-1186, 2019.
	
	\bibitem{gomes} 
	D. A. Gomes, J. Mohr, and R. R. Souza. Continuous time finite state mean field games. \emph{Applied Mathematics and Optimization}, 68(1):99-143, 2013.
	
	\bibitem{kol} V. N. Kolokoltsov. Nonlinear Markov Games on a Finite State Space (Mean-field and
	Binary Interactions). \emph{International Journal of Statistics and Probability}, 1(1): 77-91, 2012. 
	
		\bibitem{Lacker1} D. Lacker. A general characterization of the mean field limit for stochastic differential games. \emph{Probability
Theory and Related Fields}, 165(3-4):581-648, 2016.
	
	
	\bibitem{Lacker} D. Lacker. Limit theory for controlled McKean-Vlasov dynamics. \emph{SIAM Journal on Control and Optimization}, 55(3):1641-1672, 2017.
	
	\bibitem{Lacker2} D. Lacker. On the convergence of closed-loop Nash equilibria to the mean field game limit. \emph{Annals of Applied Probability},
30(4):1693-1761, 2020.
	
	
	
	\bibitem{LaurierePironneau} M. Lauri\`ere and O. Pironneau. Dynamic programming for mean field type control. \emph{Comptes
Rendus Mathematique}, ser. I, 352:707–713, 2014.

\bibitem{LauriereTangpi} M. Lauri\`ere and L. Tangpi. Convergence of large population games to mean field games with interaction through the controls. Preprint 2020, arXiv:2004.08351.


\bibitem{motte} 
	M. Motte and H. Pham. Mean-field Markov decision processes with common noise and open-loop controls. Preprint 2019,  arXiv:1912.07883.

	
	\bibitem{PhamWei} H. Pham and X. Wei. Bellman equation and viscosity solutions for mean field stochastic
control problem. \emph{ESAIM: Control, Optimization and Calculus of Variations}, 24(1):437-461, 2018.

\bibitem{PhamWei2} H. Pham and X. Wei. Dynamic programming for optimal control of stochastic McKean-Vlasov dynamics. \emph{SIAM Journal on Control and Optimization}, 55:1069-1101, 2017.
	
\bibitem{porrettaricciardi} 	
A. Porretta and M. Ricciardi. Mean field games under invariance conditions for the state space. \emph{Communications in
Partial Differential Equations}, 45(2):146-190, 2020.
	
	\bibitem{RR}
S. T. Rachev and L. Rüschendorf,
\emph{Mass Transportation Problems}, Springer-Verlag, 1998.
	
	\bibitem{Sou} P. E. Souganidis. Approximation schemes for viscosity solutions of Hamilton-Jacobi equations. 
\emph{Journal of Differential Equations}, 59: 1-43, 1985. 
	
\bigskip	

\bigskip

\end{thebibliography}
\end{document}